\documentclass{amsart}
\usepackage{amsmath}
\usepackage{amssymb}
\usepackage{amsthm}
\usepackage{dsfont}
\usepackage{lineno}
\usepackage{subfigure}
\usepackage{graphicx}
\usepackage{multirow}
\usepackage{color}
\usepackage{xspace}
\usepackage{pdfsync}
\usepackage{hyperref} 
\usepackage{algorithmicx}
\usepackage{algpseudocode}
\usepackage{algorithm}
\usepackage{mathtools}
\usepackage{enumitem}
\usepackage{epstopdf}
\DeclareMathOperator*{\argmin}{argmin}

\DeclarePairedDelimiter\floor{\lfloor}{\rfloor}
\newcommand{\bq}{\begin{equation}}
\newcommand{\eq}{\end{equation}}
\newcommand{\R}{\mathbb{R}}

\newcommand{\abs}[1]{\left\vert#1\right\vert}

\newcommand{\G}{\mathcal{G}}
\newcommand{\bO}{\mathcal{O}}
\newcommand{\dist}{\text{dist}}

\newcommand{\Dt}{\mathcal{D}}

\newcommand{\MA}{Monge-Amp\`ere\xspace}

\newcommand{\uex}{u_{\text{ex}}}
\newcommand{\cex}{c_{\text{ex}}}
\newcommand{\ulin}{\tilde{u}}

\newcommand{\itan}[1]{\tan^{-1}\left({#1}\right)}

\algnewcommand{\LineComment}[1]{\State \(\triangleright\) #1}

\newtheorem{theorem}{Theorem}

\theoremstyle{lemma}
\newtheorem{lemma}[theorem]{Lemma}
\newtheorem{lem}[theorem]{Lemma}

\newtheorem{definition}[theorem]{Definition}
\newtheorem{defn}[theorem]{Definition}

\newtheorem{remark}[theorem]{Remark}

\theoremstyle{remark}

\makeatletter
\newcommand\appendix@section[1]{%
\refstepcounter{section}%
\orig@section*{Appendix \@Alph\c@section: #1}%
}
\let\orig@section\section
\g@addto@macro\appendix{\let\section\appendix@section}
\makeatother

\begin{document}

\title[Finite difference method for minimal Lagrangian graphs]{A convergent finite difference method for computing minimal Lagrangian graphs}

\author{Brittany Froese Hamfeldt}
\address{Department of Mathematical Sciences, New Jersey Institute of Technology, University Heights, Newark, NJ 07102}
\email{bdfroese@njit.edu}
\author{Jacob Lesniewski}
\address{Department of Mathematical Sciences, New Jersey Institute of Technology, University Heights, Newark, NJ 07102}
\email{jl779@njit.edu}

\thanks{The first author was partially supported by NSF DMS-1619807 and NSF DMS-1751996. The second author was partially supported by NSF DMS-1619807.}

\begin{abstract}
We consider the numerical construction of minimal Lagrangian graphs, which is related to recent applications in materials science, molecular engineering, and theoretical physics.  It is known that this problem can be formulated as an additive eigenvalue problem for a fully nonlinear elliptic partial differential equation.  We introduce and implement a two-step generalized finite difference method, which we prove converges to the solution of the eigenvalue problem.  Numerical experiments validate this approach in a range of challenging settings.  We further discuss the generalization of this new framework to Monge-Amp\`ere type equations arising in optimal transport.  This approach holds great promise for applications where the data does not naturally satisfy the mass balance condition, and for the design of numerical methods with improved stability properties.
\end{abstract}

\date{\today}    
\maketitle

We consider the problem of constructing a diffeomorphism $f:X\to Y$ such that the graph
\bq\label{eq:manifold}
\Sigma = \{(x,f(x)) \mid x\in X\}
\eq
is a Lagrangian submanifold of $\R^n\times\R^n$ with minimal area (or equivalently, having zero mean curvature).  Here $X, Y\subset\R^n$ are smooth, convex, and bounded.  The problem of constructing minimal surfaces is important in applications such as materials science~\cite{Hoffman} and molecular engineering~\cite{Bates}.  There has also been recent interest in the use of mean curvature flows to generate minimal Lagrangian submanifolds of Calabi-Yau manifolds~\cite{Smoczyk,ThomasYau}.

Here we are interested in $\R^n\times\R^n$ equipped with the symplectic form
\bq\label{eq:2form}
\omega = \sum\limits_{i=1}^n dx_i \wedge dy_i
\eq
where coordinates in $\R^n\times\R^n$ are given by $(x_1, \ldots, x_n, y_1, \ldots, y_n)$.  A submanifold $L$ is said to be Lagrangian if 
$\omega\vert_\Sigma=0$, which is equivalent to the condition that $f$ can be expressed as a gradient: $f = \nabla u$~\cite{HarveyLawson}.

In order to compute a minimal Lagrangian submanifold, we can equivalently seek a submanifold whose Lagrangian angle (which is a primitive of mean curvature) is constant.  This can be expressed as the following eigenvalue (or additive eigenvalue) problem for a nonlinear elliptic PDE,
\bq\label{eq:LagrangianPDE}
-\sum\limits_{i=1}^n \arctan(\lambda_i(D^2u(x))) +c = 0, \quad x \in X
\eq
where $\lambda_i(D^2u)$ denote the eigenvalues of the Hessian of $u$ and the constant $c$ is not known \emph{a priori}.  This is augmented by the so-called second type boundary condition
\bq\label{eq:BVP2}
\nabla u(X) = Y.
\eq
A result by Brendle and Warren~\cite{brendle2010} showed that as long as $X$ and $Y$ are uniformly convex,~\eqref{eq:LagrangianPDE}-\eqref{eq:BVP2} has a unique (up to additive constants) convex solution~$u\in C^2(X)$ with associated Lagrangian angle $c\in(0,\frac{n\pi}{2})$.

Equation~\eqref{eq:LagrangianPDE} is an example of a fully nonlinear elliptic partial differential equation (PDE).  The past few years have seen a rising interest in numerical techniques for solving fully nonlinear elliptic PDEs, with several new approaches being introduced including~\cite{BrennerNeilanMA2D,DGnum2006,FengNeilan,FroeseMeshfreeEigs,ObermanEigenvalues}.  The unusual boundary condition~\eqref{eq:BVP2} has also received recent attention because of its relationship to optimal transport~\cite{Benamou:2014:NSO:2574571.2574621,Prins_BVP2}.  

The PDE for minimal Lagrangian submanifolds is unique, however, in that it involves an additional unknown constant $c$.  In fact, \MA equations and other Generated Jacobian Equations related to Optimal Transport may also be expressed this way, as hinted at in~\cite{FroeseTransport}.  There are distinct advantages to using this formulation in applications where data does not naturally satisfy the mass balance condition (e.g., image registration~\cite{Haker}, seismic full waveform inversion~\cite{EF_FWI}, mesh generation~\cite{BuddMeshGen}) and in problems where consistent discretizations fail to inherit the well-posedness of the underlying PDE.

In this article, we develop a framework for numerically solving eigenvalue problems of the form
\bq\label{eq:EigProblem}
F(x, u(x), \nabla u(x), D^2u(x)) + cf(x,u(x), \nabla u(x)) = 0
\eq
where $F$ is an elliptic operator and the PDE is coupled to a second type boundary condition~\eqref{eq:BVP2}. In particular, we use this framework to introduce and implement a numerical method for computing minimal Lagrangian submanifolds.  The method utilizes generalized finite difference approximations on augmented piecewise-Cartesian grids, as introduced in~\cite{HS_Quadtree}.  We adapt this to the minimal Lagrangian problem, and show how this approach can be used to enforce the second boundary condition~\eqref{eq:BVP2} and numerically compute the eigenvalue $c$. Though the method is implemented in two dimensions, it could be easily adapted to higher dimensions and more complicated PDEs.

We prove that our method converges to the solution of the nonlinear eigenvalue problem~\eqref{eq:LagrangianPDE}-\eqref{eq:BVP2}. Ultimately, the techniques developed and analyzed for the minimal Lagrangian problem hold great promise for the solution of other more challenging PDEs related to Optimal Transport.

\section{Background}\label{sec:background}

\subsection{Elliptic equations}
A PDE 
\bq\label{eq:PDE} 
G(x,u,\nabla u,D^2 u) = 0 
\eq
is fully nonlinear and elliptic if it exhibits nonlinear dependence on the highest order derivative, and satisfies the ellipticity condition: 
\begin{defn}[elliptic operator]The PDE~\eqref{eq:PDE} is \textbf{(degenerate) elliptic} if $$G(x,r,p,A) \leq G(x,s,p,B)$$ for all $x \in \bar{X}$, $r,s \in \R$, $p\in\R^n$, $A,B \in S^n$ with $A \geq B$ and $r \leq s$ where $A \geq B$ means $A - B$ is a positive definite matrix and $S^n$ is the set of symmetric $n\times n$ matrices. \label{def:ELC}
\end{defn} 

We remark that the PDE operator is not required to be continuous in space.  In particular, this allows us to incorporate boundary conditions directly into the operator $G$.  In the present article, we will be particularly interested in problems where the boundary operator takes the form
\bq\label{eq:Gbdy}
G(x,u,\nabla u,D^2u) \equiv H(x,\nabla u), \quad x\in\partial X
\eq
and the boundary operator $H(x,\nabla u)$ can be written in terms of one-sided directional derivatives.

A desirable property that is shared by many elliptic operators is the comparison principle.
\begin{defn}[comparison principle]\label{def:comparison}
The PDE operator~\eqref{eq:PDE} satisfies a \textbf{comparison principle} if whenever $G(x,u(x),\nabla u(x),D^2u(x)) \leq G(x,v(x),\nabla v(x),D^2v(x))$ for all $x\in \bar{\Omega}$ then $u(x) \leq v(x)$ for all $x \in \bar{X}$.
\end{defn}

A comparison principle can be used to establish existence and uniqueness of solutions to the PDE.  A common technique for proving existence is Perron's method, which involves arguing that the maximal subsolution
\bq\label{eq:maxSub}
u(x) \equiv \sup\left\{v(x) \mid G(x,v(x),\nabla v(x),D^2v(x)) \leq 0\right\}
\eq
is actually a solution to the PDE.  Uniqueness of solutions follows immediately from a comparison principle.

Many fully nonlinear elliptic equations do not possess a classical solution, and thus some notion of weak solution is needed.  A powerful approach is the viscosity solution, which relies on a maximum principle argument to transfer derivatives onto smooth test functions~\cite{CIL}.

In order to define viscosity solutions, we first must define the upper and lower semicontinuous envelopes.
\begin{defn}[upper and lower semicontinuous envelopes]\label{def:envelopes} The \textbf{upper and lower semicontinuous envelopes} of a function u(x) are defined by
$$u^*(x) = \limsup_{y\to x} u(y),$$
and
$$u_*(x) = \liminf_{y \to x} u(y)$$  respectively. 
\end{defn}

\begin{defn}[viscosity solution]\label{def:viscosity}
An upper(lower) semicontinuous function $u$ is a \textbf{viscosity sub(super)solution} of~\eqref{eq:PDE} if for any $x \in \Omega$ and any $\phi \in C^{2}(\Omega)$ such that $u - \phi$ attains a local maximum(minimum) at $x$, 
\[G_*^{(*)}(x,u(x),D\phi(x),D^2 \phi(x)) \leq(\geq) 0.\]
A continuous function is a \textbf{viscosity solution} of~\eqref{eq:PDE} if it is both a viscosity sub- and supersolution.
\end{defn}

Viscosity solutions provide a framework which allows many comparison, uniqueness, existence, and continuous dependence theorems to be proved. An equation can be shown to have a unique viscosity solution if it has a comparison principle~\cite{CIL}.

\subsection{Eigenvalue problem for a PDE}\label{sec:eig}
The equation~\eqref{eq:LagrangianPDE}-\eqref{eq:BVP2} we consider in this article is an example of an eigenvalue problem for a fully nonlinear elliptic operator.  Abstractly, the problem statement is to find $u \in C^2(X)\cap C^1(\bar{X})$ and $c\in\R$ such that
\bq\label{eq:eigProblemAbstract}
\begin{cases}
F(x,\nabla u(x),D^2u(x)) + cf(x,\nabla u(x)) = 0, & x \in X\\
H(x,\nabla u(x)) = 0, & x \in \partial X.
\end{cases}
\eq
We remark that the solution is at best unique only up to additive constants.

In fact, this formulation of the problem is intricately connected to the solvability of a related PDE.  As an example, we consider the Neumann problem for Poisson's equation.
\bq\label{eq:poisson}
\begin{cases}
-\Delta u + f = 0, & x\in \Omega\\
\frac{\partial u}{\partial n} = g , & x\in\partial \Omega
\end{cases}
\eq

For a solution to exist, data must satisfy the solvability condition
$$\int_\Omega f({x}) \,dx = \int_{\partial \Omega} g({x})\,dS\text{.}$$
However, data $f, g$ arising in applications are susceptible to noise, measurement error, etc. This can lead to a failure in the solvability condition. One approach to ensuring solvability in this case is to relax the problem and interpret it as an (additive) eigenvalue problem by introducing a constant $c$ and solving 
\bq\label{eq:poissonEig}
\begin{cases}
-\Delta u + c f = 0, & x\in \Omega\\
\frac{\partial u}{\partial n} = g, & x \in \partial\Omega.
\end{cases}
\eq
for the unknown pair $(u,c)$.

The new solvability condition is 
\bq\label{eq:solvePoisson} c\int_\Omega f(x) \,dx = \int_{\partial \Omega} g({x})\,dS\text{.}\eq
The solution of the eigenvalue problem~\eqref{eq:poissonEig} will then select a value of $c$ that satisfies this condition and forces the problem to be solvable.  If $f$ and $g$ are close to satisfying the solvability condition, then the solution will choose $c\approx 1$ and produce a solution to a PDE close to the original~\eqref{eq:poisson}, with the error due to errors in the input data $f, g$.

A similar issue arises in the solution of the second boundary value problem for the \MA equation, which arises in the context of optimal transport.
\bq\label{eq:MA}
\begin{cases}
-g(\nabla u(x))\det(D^2u(x)) + f(x) = 0, \quad x \in X\\
u \text{ is convex}\\
\nabla u(X) = Y.
\end{cases}
\eq
This problem has a solution only if the following mass balance condition is satisfied,
\bq\label{eq:massBalance}
\int_X f(x)\,dx = \int_Y g(y)\,dy.
\eq
However, in many applications (e.g., image processing~\cite{Haker}, seismic full waveform inversion~\cite{EF_FWI}, mesh generation~\cite{BuddMeshGen}, etc.) the data is not expected to naturally satisfy the solvability condition.  A proposed solution is to view the equation as an eigenvalue problem and seek a pair $(u,c)$ satisfying
\bq\label{eq:MAEig}
\begin{cases}
-g(\nabla u(x))\det(D^2u(x)) + cf(x) = 0, \quad x \in X\\
u \text{ is convex}\\
\nabla u(X) = Y.
\end{cases}
\eq

In fact, even when data does satisfy the relevant solvability condition, consistent discretizations of~\eqref{eq:poisson} or~\eqref{eq:MA} cannot be expected to inherit this solvability.
To illustrate this, consider Poisson's equation in one dimension with Neumann boundary conditions:
$$
\begin{cases}
-u''(x) + f(x) = 0  & x \in (0,1) \\
u'(x) = g(x)  & x = 0,1
\end{cases}
$$
where
\[f(x) \equiv \cos{\left(\frac{\pi}{2}x\right)}, \quad g(x) \equiv \frac{2}{\pi}\sin{\left(\frac{\pi}{2} x\right)}  \]
This has a solution, which is unique up to additive constants, since the data satisfies the solvability condition
$$\int_{0}^{1} f(x) dx = g(1) - g(0).$$ 

Now consider the uniform grid $x_j = jh$, $j= 0, \ldots, N$ and discretize the equation using standard centered differences for the second derivative and a one-sided difference for the boundary condition.  It is not hard to check that the resulting linear system has a solution only if the following discrete solvability condition is satisfied~\cite{LeVeque:2007:FDM:1355322}:
\[ h\sum\limits_{j=1}^{N-1}f(x_j) = g(x_N)-g(x_0). \]
This is a natural discrete analogue of the continuous solvability condition, but it is not exactly satisfied at the discrete level and the discrete problem thus fails to have a solution.

As an alternative, we view the Poisson equation as the following eigenvalue problem.
$$
\begin{cases}
-u''(x) + c f(x) = 0  & x \in (0,1) \\
u'(x) = g(x)  & x = 0,1
\end{cases}
$$

We discretize as before, including the eigenvalue $c$ as an additional unknown, and supplementing the linear system with an additional equation $u(x_0) = 0$ in order to select a unique solution.  This time, the discrete problem has a solution $u^h$ with corresponding eigenvalue $c^h$.  We verify that both $u^h\to u$ and $c^h\to1$ as the grid is refined, so that the limiting problem is the original Poisson equation.  See Figure~\ref{fig:poissonEx}.

\begin{figure}
\centering
\includegraphics[width=0.5\textwidth]{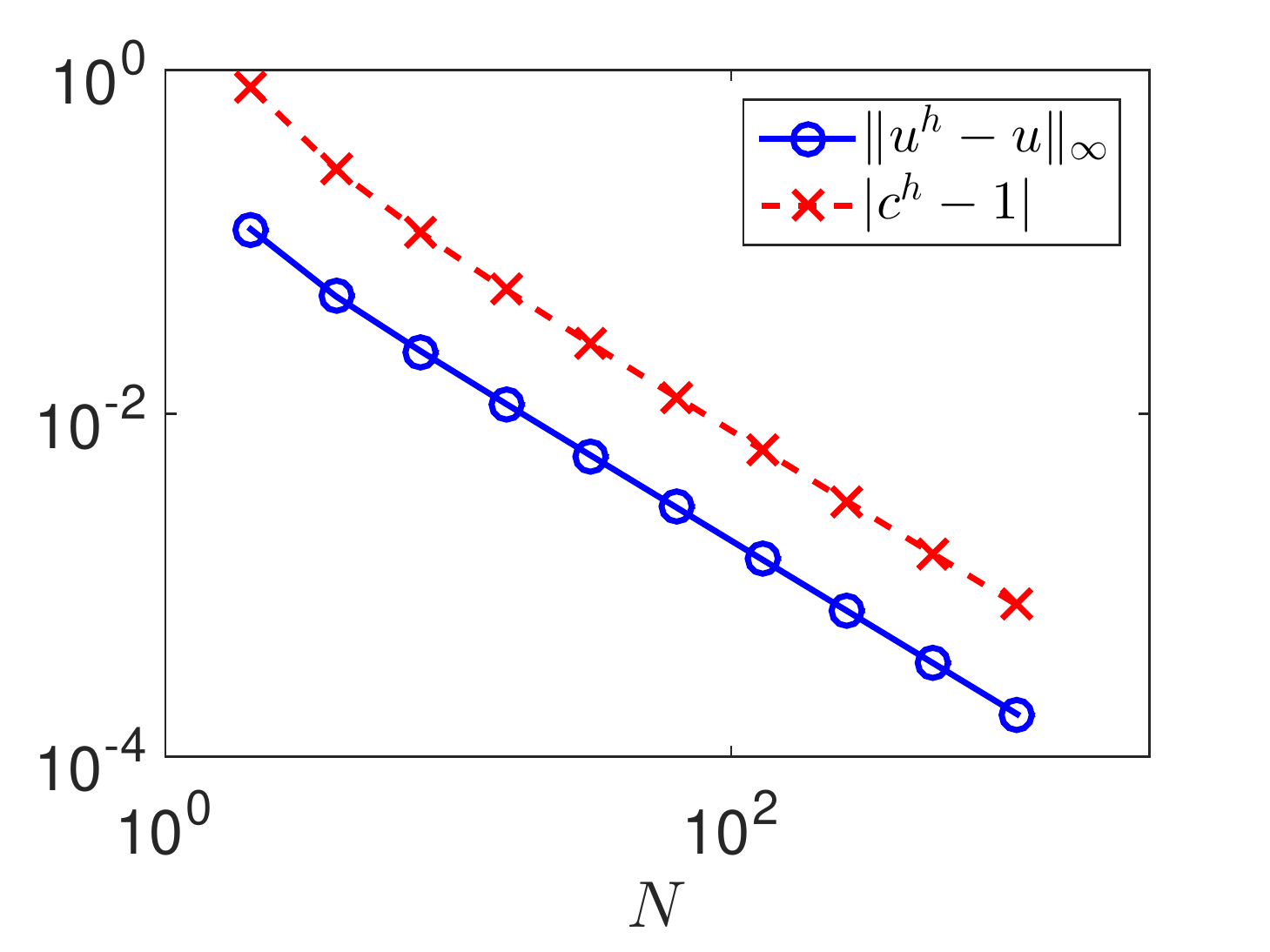}
\caption{Discrete solution to Poisson's equation when viewed as an eigenvalue problem.}
\label{fig:poissonEx}
\end{figure}

There is certainly a need for numerical methods and convergence analysis that can be applied to eigenvalue problems for fully nonlinear elliptic equations.  In addressing this issue for the construction of minimal Lagrangian graphs, we also begin the development of a framework for solving many other important nonlinear PDEs.

\subsection{Second boundary value problem}\label{sec:bvp2}

The unusual second type boundary condition~\eqref{eq:BVP2} does not at first glance appear to be a boundary condition at all.  However, when $u$ is a convex function (which is the case for our problem~\cite{brendle2010}) and $X, Y$ are uniformly convex sets, it can be recast as a nonlinear Neumann type boundary condition.  To do this, we require a defining function for the target set $Y$~\cite{Delanoe,Urbas_BVP2}, which should have the property that
\[
H(y) = \begin{cases} <0, & y \in Y\\ = 0, & y\in\partial Y\\>0, & y\notin\bar{Y}.\end{cases}
\]
A natural choice is the signed distance function to the target boundary $\partial Y.$

In this case, as in~\cite{Benamou:2014:NSO:2574571.2574621}, we can rewrite the boundary condition as
\bq\label{eq:HJBC}
H(\nabla u(x)) = 0, \quad x \in \partial X.
\eq
This simply requires that all points on the boundary of the domain $X$ are mapped (via the gradient of $u$) onto the boundary of the target set $Y$.

Then the problem of constructing minimal Lagrangian graphs can be recast as the following eigenvalue problem with nonlinear Neumann type boundary conditions.
\bq\label{eq:eigHJBC}
\begin{cases}
F(D^2u(x)) + c = 0, & x \in X\\
H(\nabla u(x)) = 0, & x \in \partial X
\end{cases}
\eq
where
\bq\label{eq:arctanOp}
F(D^2u(x)) = -\sum\limits_{i=1}^n\arctan(\lambda_i(D^2u(x)))
\eq
and
\bq\label{eq:HJOp}
H(\nabla u(x)) = \begin{cases} -\dist(\nabla u(x),\partial Y),&  \nabla u(x) \in Y\\
0, & \nabla u(x) \in \partial Y\\ \dist(\nabla u(x), \partial Y), & \nabla u(x) \notin\bar{Y}. \end{cases}
\eq

From the convexity of $Y$, the signed distance function to its boundary is also convex, and thus, $H$ is convex. We can rewrite $H$ in terms of supporting hyperplanes to the convex target set 
\begin{align}H(y) = \sup_{y_0\in\partial Y}\{n(y_0)\cdot (y - y_0)\}\end{align}
 where $n(y_0)$ is the outward normal to $\partial Y$ at $y_0$ \cite{Benamou:2014:NSO:2574571.2574621}. By duality, this is equivalent to 
\begin{align}H(y) = \sup_{\abs{n}=1}\{n\cdot (y - y_0(n)) \}\end{align}
 where $y_0(n)$ is the point on the boundary of $Y$ with the normal $n$. Then if $n$ is a unit outward normal to $Y$ at $y_0$, the Legendre-Fenchel transform of $H(y)$ is
\begin{align}H^*(n) =\sup_{y_0\in\partial Y} \{ n\cdot y_0 - H(y_0)\} = \sup_{y_0\in\partial Y} \{n\cdot y_0\} \end{align}
 and we can rewrite the condition as~\cite{Benamou:2014:NSO:2574571.2574621} 
\bq\label{eq:HJLF}H(y) = \sup_{\abs{n}=1}\{n\cdot y - H^*(n) \}. \eq
\begin{lemma}\label{lem:HJOpHyperplane}
Let $u\in C^2(\bar{X})$ be uniformly convex with $X$ and $Y$ convex. Then there exists $\ell > 0$ such that for all $x \in \partial X$
\bq\label{eq:bcpos} H(\nabla u(x)) = \max_{n \cdot n_x>\ell}\{\nabla u \cdot n - H^*(n)\} \eq
\end{lemma}
From Section 2.3 of~\cite{BFOnumSolSBVP}, we know that
$$n_x \cdot n_y = a(x)\nabla H(\nabla u(x))^T D^2u(x) \nabla H(\nabla u(x)),$$
which is positive for all $x$ since $D^2u(x)$ is positive definite.  Moreover, this is continuous on the compact set $\partial X$. Thus, it has a minimum, which must also be a positive value $\ell$.
From the same section in \cite{BFOnumSolSBVP}, we know that the maximum in~\eqref{eq:bcpos} is attained when $n = n_y$. Since we know that $n_x \cdot n_y \geq \ell>0$, we can restrict the maximum to vectors $n$ satisfying this constraint.

\subsection{Discretization of elliptic PDEs}\label{sec:discBackground}
In order to build convergent methods for the eigenvalue problem~\eqref{eq:LagrangianPDE}, we wish to build upon recent developments in the approximation of fully nonlinear elliptic equations.

Classically, the convergence of numerical methods is established via the Lax-Equivalence Theorem.  Roughly speaking, a consistent, stable method will converge to the solution of the continuous equation.  However, this does not immediately yield convergent methods for fully nonlinear equations for a couple of reasons. First, establishing the existence and stability of solutions to a discrete method can be a delicate problem in the case of nonlinear equations and secondly, it does not apply when the equation does not have classical solutions.

A powerful contribution to the numerical approximation of elliptic equations was provided by the Barles-Souganidis framework, which states that the solution to a scheme that is consistent, monotone, and stable will converge to the viscosity solution, provided the underlying PDE satisfies a comparison principle~\cite{BSnum}.

In this article, we consider finite difference schemes that have the form
\bq\label{eq:approx} G^h(x,u(x),u(x)-u(\cdot)) = 0, \quad x\in \G^h \eq
where $u:\G^h\to\R$ is a grid function and $\G^h\subset\bar{X}$ is a finite set of discretization points, which can be a finite difference grid or a more general point cloud.  Here $h$ is a small parameter relating to the grid resolution. In particular, we expect that as $h\to0$, the domain becomes fully resolved in the sense that
\bq\label{eq:resolution} \lim\limits_{h\to0} \sup\limits_{y\in\Omega}\min\limits_{x\in\G^h}\abs{x-y} = 0. \eq

In this setting, the properties required by the Barles-Souganidis framework can be defined as follows.

\begin{definition}[Consistency]\label{def:consistency}
The scheme~\eqref{eq:approx} is \textbf{consistent} with the PDE operator~\eqref{eq:PDE}
 if for any smooth function $\phi$ and $x\in\bar{\Omega}$,
\[ \lim_{h\to0, y\in\G^h\to x} G^h(y,\phi(y),\phi(y)-\phi(\cdot)) = G(x,\phi(x),\nabla\phi(x),D^2\phi(x)). \]
\end{definition}

To consistent schemes, we also associate a truncation (consistency) error $\tau(h)$ .
\begin{definition}[Truncation error]\label{def:truncation}
The truncation error $\tau(h) > 0$ of the scheme~\eqref{eq:approx} is a quantity chosen so that for every smooth function $\phi$
\[ \limsup\limits_{h\to0}\max\limits_{x\in\G^h}\frac{\abs{G^h(x,\phi(x),\phi(x)-\phi(\cdot)) - G(x,\phi(x),\nabla\phi(x),D^2\phi(x))}}{\tau(h)} < \infty. \]
\end{definition}


\begin{definition}[Monotonicity]\label{def:monotonicity}
The scheme~\eqref{eq:approx} is \textbf{monotone} if $G^h$ is a non-decreasing function of its final two arguments.
\end{definition}

\begin{definition}[Stability]\label{def:stability}
The scheme~\eqref{eq:approx} is \textbf{stable} if there exists a constant $M$, independent of $h$, such that if $h>0$ and $u^h$ is any solution of~\eqref{eq:approx} then $\|u^h\|_\infty \leq M$.
\end{definition}

\begin{definition}[Continuity]\label{def:continuous}
The scheme~\eqref{eq:approx} is \textbf{continuous} if $G^h$ is continuous in its last two arguments.
\end{definition}

The Barles-Souganidis convergence framework does not apply to all elliptic PDEs, including~\eqref{eq:eigHJBC}, which does not have the required comparison principle.  Nevertheless, it provides an important starting point for the development of convergent numerical methods.  

In particular, monotone schemes possess some form of a comparison principle even if the limiting PDE does not.  Under the additional assumption that $G^h$ is strictly increasing in its second argument, we can obtain a discrete comparison principle very similar to Definition~\ref{def:comparison}, which in turn yields uniqueness of solutions to the approximation scheme; see~\cite[Theorem~5]{ObermanEP}.  The schemes we consider in this work do not satisfy this traditional comparison principle.  Nevertheless, as noted in~\cite[Lemma~5.4]{Hamfeldt_Gauss}, they do obey a slightly weaker form of comparison principle.

\begin{lem}[Discrete comparison principle]\label{lem:discreteComp}
Let $G^h$ be a monotone scheme and $G^h(x,u(x),u(x)-u(\cdot)) < G^h(x,v(x),v(x)-v(\cdot))$ for every $x\in\G^h$.  Then $u(x) \leq v(x)$ for every $x\in\G^h$.
\end{lem}

\begin{remark}
Because the inequality in this discrete comparison principle is strict, it does not guarantee solution uniqueness.  Moreover, for some monotone schemes (including those described in the present article), it is not possible to find grid functions $u, v$ such that
$G^h(x,u(x),u(x)-u(\cdot)) < G^h(x,v(x),v(x)-v(\cdot))$ at every grid point.  This observation allows us to use the discrete comparison principle as a key element in a proof by contradiction argument.
\end{remark}

The proof of Lemma~\ref{lem:discreteComp} is essentially identical to the proof of~\cite[Theorem~5]{ObermanEP}, but is included here for completeness.

\begin{proof}[Proof of Lemma~\ref{lem:discreteComp}]
We suppose that $G^h(x,u(x),u(x)-u(\cdot)) < G^h(x,v(x),v(x)-v(\cdot))$ for every $x\in\G^h$ and choose $y\in\G^h$ such that
\[ u(y)-v(y) = \max\limits_{x\in\G^h}\left\{u(x)-v(x)\right\}, \]
which is well-defined since $\G^h$ is a finite set.  In particular, this yields
\[ u(y)-u(x) \geq v(y)-v(x), \quad x\in\G^h. \]

Now we suppose that
\[ u(y)-v(y)>0. \]
By monotonicity, we find that
\begin{align*}
G^h(y,u(y),u(y)-u(\cdot)) &\geq G^h(y,v(y),v(y)-v(\cdot))\\
  &> G^h(y,u(y),u(y)-u(\cdot)),
\end{align*}
where the last step is simply the hypothesis of the lemma.  This is a contradiction, and we conclude that
\[ \max\limits_{x\in\G^h}\left\{u(x)-v(x)\right\} \leq 0. \]
\end{proof}

Our goal in this article is to exploit the discrete comparison principle to prove that computed eigenvalues $c^h$ converge to the exact eigenvalue of~\eqref{eq:eigHJBC}.  From there, we introduce additional stability into our scheme, which allows us to modify the Barles-Souganidis argument to prove convergence of the computed solution $u^h$ even in the absence of a comparison principle.

\section{Reformulation of the PDE}\label{sec:pde}
We begin by proposing a reformulation of the PDE~\eqref{eq:eigHJBC}, which will allow us to build more stability into our numerical schemes.  Moreover, we demonstrate that viscosity solutions of this new equation (with the eigenvalue $\cex$ fixed) are equivalent to classical solutions of the original problem.

We remark first of all that solutions to the second boundary condition~\eqref{eq:BVP2} will trivially satisfy \emph{a priori} bounds on the solution gradient.  That is, choose any $R>\max\{\abs{p} \mid p \in \partial Y\}$
 and let $u$ satisfy the second boundary condition~\eqref{eq:BVP2}.  Then 
\bq\label{eq:gradBound} \abs{\nabla u(x)} < R \eq
for all $x \in \bar{X}$.

We also recall that any smooth convex solution of the second boundary condition, reformulated as in~\eqref{eq:HJBC}, will satisfy the constraints
\bq\label{eq:constraints}
\begin{aligned}
-\lambda_1(D^2u(x)) &\leq 0\\
H(\nabla u(x)) & \leq 0
\end{aligned}
\eq
for every $x\in X$.  Here $\lambda_1(M)$ denotes the smallest eigenvalue of the symmetric positive definite matrix $M$.

We propose combining all of these constraints into a new PDE
\bq\label{eq:modifiedPDE}\max\left\{F(D^2u(x)) + \cex, -\lambda_1(D^2u(x)),H(\nabla u(x)),\abs{\nabla u(x)} - R\right\} = 0, \quad x \in X. \eq
We remark that this equation is posed only in the interior of the domain, and boundary conditions will not be required to select a unique (up to additive constants) solution.  We also note that in the above equation, the eigenvalue $\cex$ will be interpreted as a known quantity.

\begin{theorem}[Equivalence of PDEs]\label{thm:equiv}
Let $u:\bar{X}\to\R$ be continuous and $\cex\in(0,n\pi/2)$ be the unique eigenvalue of~\eqref{eq:eigHJBC}.  Then $(u,\cex)$ is a classical solution of~\eqref{eq:eigHJBC} if and only if $u$ is a viscosity solution of~\eqref{eq:modifiedPDE}.
\end{theorem}
\begin{proof}
This result is an immediate consequence of Lemmas~\ref{lem:equiv1}-\ref{lem:equiv2}, proved below.
\end{proof}

\begin{lemma}[Classical implies viscosity]\label{lem:equiv1}
Let  $(u,\cex)$ be a classical solution of~\eqref{eq:eigHJBC}. Then $u$ is a viscosity solution of~\eqref{eq:modifiedPDE}.
\end{lemma}
\begin{proof}
We remark that $u$ trivially satisfies the constraints~\eqref{eq:gradBound}-\eqref{eq:constraints}.  Since additionally 
\[ F(D^2u(x))+c_{ex} = 0, \]
it is certainly true that the modified equation~\eqref{eq:modifiedPDE} holds in the classical sense. It is a simple consequence that~\eqref{eq:modifiedPDE} will also hold in the viscosity sense~\cite{CIL}.
\end{proof}

\begin{lemma}[Viscosity implies classical]\label{lem:equiv2}
Let $u:\bar{X}\to\R$ be continuous and $\cex\in(0,n\pi/2)$ be the unique eigenvalue for~\eqref{eq:eigHJBC}.  If $u$ is a viscosity solution of~\eqref{eq:modifiedPDE} then $(u,\cex)$ is a classical solution of~\eqref{eq:eigHJBC}.
\end{lemma}
\begin{proof}
Let $\uex$ be any classical solution of~\eqref{eq:eigHJBC}.  From~\cite{brendle2010}, this is uniquely determined up to an additive constant.

We remark first of all that $u$ is a viscosity subsolution of the equation
\[ -\lambda_1(D^2u(x)) = 0. \]
From~\cite[Theorem~1]{ObermanCE}, $u$ is convex.

We also observe that $u$ is a convex viscosity subsolution of the equation
\[ H(\nabla u(x)) = 0. \]
From~\cite[Lemma~2.5]{HamfeldtBVP2}, the subgradient of $u$ satisfies
\[ \partial u(X) \subset \bar{Y}. \]
As $u$ is continuous up to the boundary, a consequence of this is that
\bq\label{eq:subgradBdy} \partial u(x) \cap \bar{Y} \neq \emptyset \eq
for every $x\in\partial X$.
 
 We now assume that $u-\uex$ is not a constant and show that this leads to a contradiction.  Since $\uex$ is a viscosity solution of the constrained PDE~\eqref{eq:modifiedPDE}, it is also a subsolution of the uniformly elliptic component
\[ F(D^2u(x))+\cex \leq 0.\]
Since $\uex$ is a classical solution of 
\[ F(D^2\uex(x))+\cex = 0, \]
it is also a viscosity solution~\cite{CIL} and a viscosity supersolution.

From~\cite[Theorem~3.1]{JensenMax}, the maximum of $u-\uex$ must be attained at some point $x_0\in\partial X$.  Moreover, by a nonlinear version of the Hopf boundary lemma~\cite{lian2020boundary}, we have that
\[ \frac{\partial (u-\uex)(x_0)}{\partial n} > 0\]
for any exterior direction $n$ satisfying $n\cdot n_x(x_0) > 0$.  That is, taking any $p\in\partial u(x_0)$, we must have
\[ (p-\nabla\uex(x_0))\cdot n > 0. \]
Now we consider in particular the choice of $n = \nabla H(\nabla\uex(x_0))$, which does satisfy the requirement $n\cdot n_x(x_0) >0$ as in Lemma~\ref{lem:HJOpHyperplane}.  Hence, 
$$(p-\nabla \uex(x_0))\cdot \nabla H(\nabla \uex(x_0)) > 0$$
On the other hand, since $H$ is convex, we know that 
$$H(p) \geq H(\nabla \uex(x_0)) + \nabla H(\nabla \uex(x_0))\cdot(p-\nabla \uex(x_0)) > H(\nabla \uex(x_0)) = 0.$$
The condition $H(p) > 0$ implies $p$ is outside $\bar{Y}$ for any $p\in\partial u(x_0)$, which contradicts~\eqref{eq:subgradBdy}.

We conclude that $u-\uex$ must be constant on $X$.  Since the classical solution of~\eqref{eq:eigHJBC} is unique up to additive constants, $u$ is a classical solution.
\end{proof}

\section{Numerical Method}\label{sec:numerics}

In this section, we describe our approach to numerically solving the eigenvalue problem~\eqref{eq:LagrangianPDE}-\eqref{eq:BVP2}.
Ultimately, we will establish convergence of this method (Theorems~\ref{thm:convEig}-\ref{thm:convSol}).

\subsection{Numerical framework}
The computational and convergence framework we employ involves a two-step approach.  Let us first suppose that we have discrete approximations $F^h,H^h,E^h,L^h$ of the PDE operators $F(D^2u), H(\nabla u), \abs{\nabla u}, -\lambda_1(D^2u)$.  The details of these discrete operators will be explained in the following subsections.  These are assumed to have a maximum truncation error of $\tau(h)$ as defined in Definition~\ref{def:truncation}.  We also let $x_0 \in \bar{X}$ be any fixed point in the domain and choose a sequence $x_0^h\in \G^h$ such that $x_0^h\to x_0$.  Finally, we choose some $\kappa(h)\geq 0$.  

We now employ a two-step procedure to solve for an approximation $(u^h,c^h)$ to the true solution $(\uex,\cex)$. 

\begin{enumerate}
\item[1.] Solve the discrete system
\bq\label{eq:disc1}
\begin{cases}
F^h(x,v^h(x)-v^h(\cdot)) + c^h = 0, & x \in \G^h\cap X\\
H^h(x,v^h(x)-v^h(\cdot)) = 0, & x \in \G^h\cap\partial X\\
v^h(x_0^h) = 0
\end{cases}
\eq
for the grid function $v^h$ and scalar $c^h$.
\item[2.] Solve the discrete system
\bq\label{eq:disc2}
\begin{cases}
\max\left\{F^h(x,w^h(x)-w^h(\cdot)) + c^h,L^h(x,w^h(x)-w^h(\cdot))\right.,\\
\left.\phantom{11111}H^h(x,w^h(x)-w^h(\cdot)),E^h(x,w^h(x)-w^h(\cdot))-R\right\} = 0, & x \in \G^h\cap X\\
\max\{H^h(x,w^h(x)-w^h(\cdot)) + \kappa(h) w^h(x),E^h(x,w^h(x)-w^h(\cdot))-R\} = 0, & x \in \G^h\cap\partial X
\end{cases} 
\eq
for the grid function $w^h$ and set
\bq\label{eq:uh} u^h(x) = w^h(x)-w^h(x_0^h). \eq
\end{enumerate}

We remark that while the second step is important for the convergence analysis, we do not find it necessary to solve this second system in practice.  Instead, we typically find that the solution $v^h$ obtained in step 1 automatically satisfies the second system with $\kappa(h) = 0$.  If this does not occur, solving the second system~\eqref{eq:disc2} becomes necessary.  In that case, we should choose   $\kappa(h) > 0$ to guarantee existence of a solution.  This relaxation of the boundary condition is needed since the solvability conditions for~\eqref{eq:disc1} and~\eqref{eq:disc2} may differ slightly.

We also observe that the final candidate solution $u^h$ that we compute satisfies the scheme
\bq\label{eq:uhInt}
\begin{aligned}
G^h(x,u^h(x)-u^h(\cdot))  &\equiv
\max\left\{\right.\left.F^h(x,u^h(x)-u^h(\cdot)) + c^h,L^h(x,u^h(x)-u^h(\cdot)),\right.\\
&\phantom{====}\left.H^h(x,u^h(x)-u^h(\cdot)),E^h(x,u^h(x)-u^h(\cdot))-R \right\} = 0
\end{aligned}
\eq
at interior points $x\in\G^h\cap X$ and satisfies the inequality
\bq\label{eq:uhBdy}
E^h(x,u^h(x)-u^h(\cdot))-R \leq 0
\eq
at all points $x\in\G^h$.

The approximation schemes will have to satisfy consistency and monotonicity conditions in order to fit within the requirements of our ultimate convergence theorems (Theorems~\ref{thm:convEig}-\ref{thm:convSol}), with some additional structure built into the discrete Eikonal operator~$E^h$.

\subsection{Quadtree meshes}
We begin by describing the meshes we use to discretize the PDE.  It is possible to construct convergent methods on very general meshes or point clouds.  However, 
 we desire a mesh with the flexibility to resolve directional derivatives in many directions and deal with complicated geometries, while retaining enough structure to allow for an efficient implementation.  For this reason, we choose to utilize piecewise Cartesian meshes augmented with additional nodes along the boundary.  These can be conveniently stored using a quadtree structure as in~\cite{HS_Quadtree}.  See Figure~\ref{fig:quadtree} for examples of such meshes.

\begin{figure}[ht]
\centering
\subfigure[]{
\includegraphics[width=0.45\textwidth]{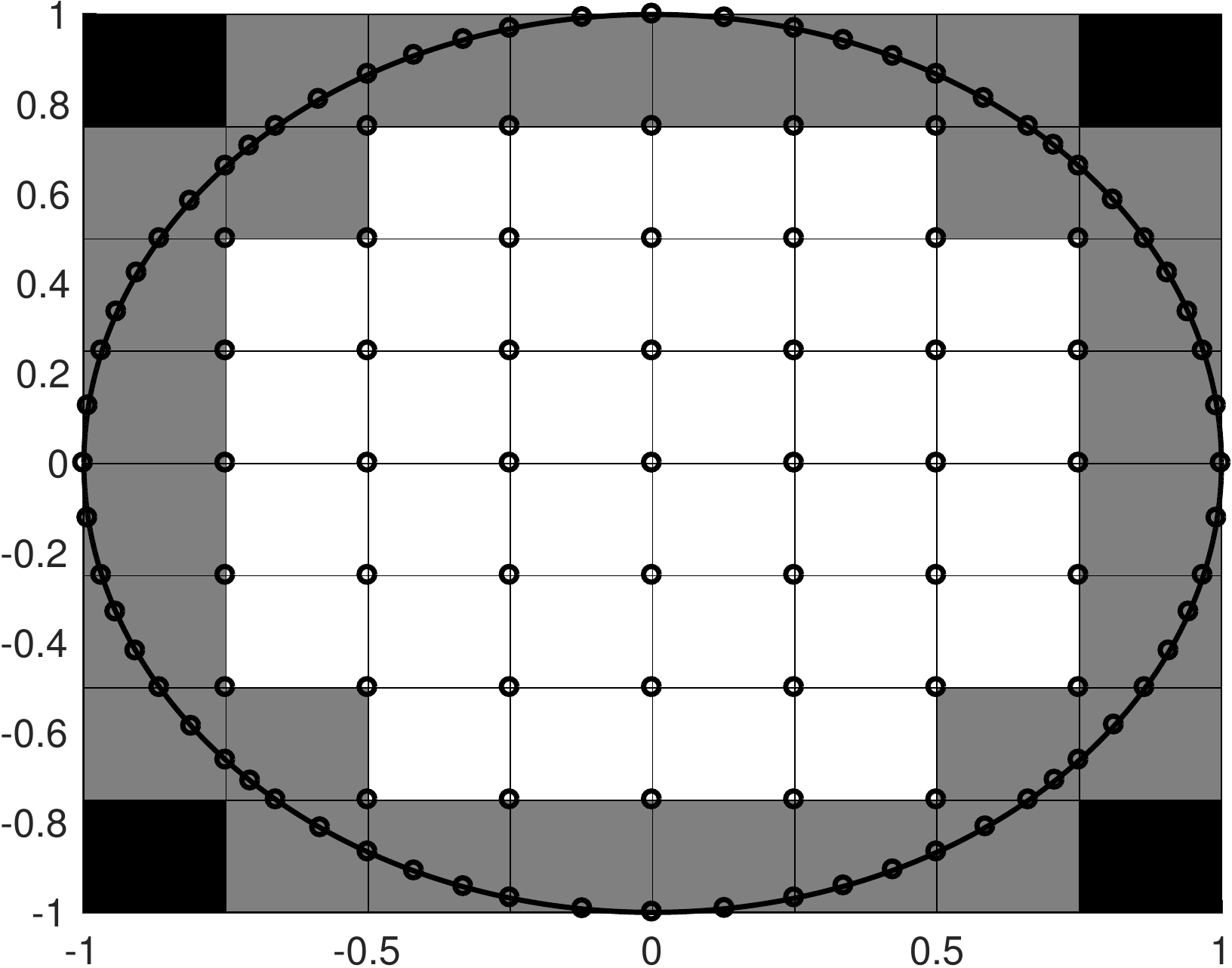}\label{fig:quad1}}
\subfigure[]{
\includegraphics[width=0.45\textwidth]{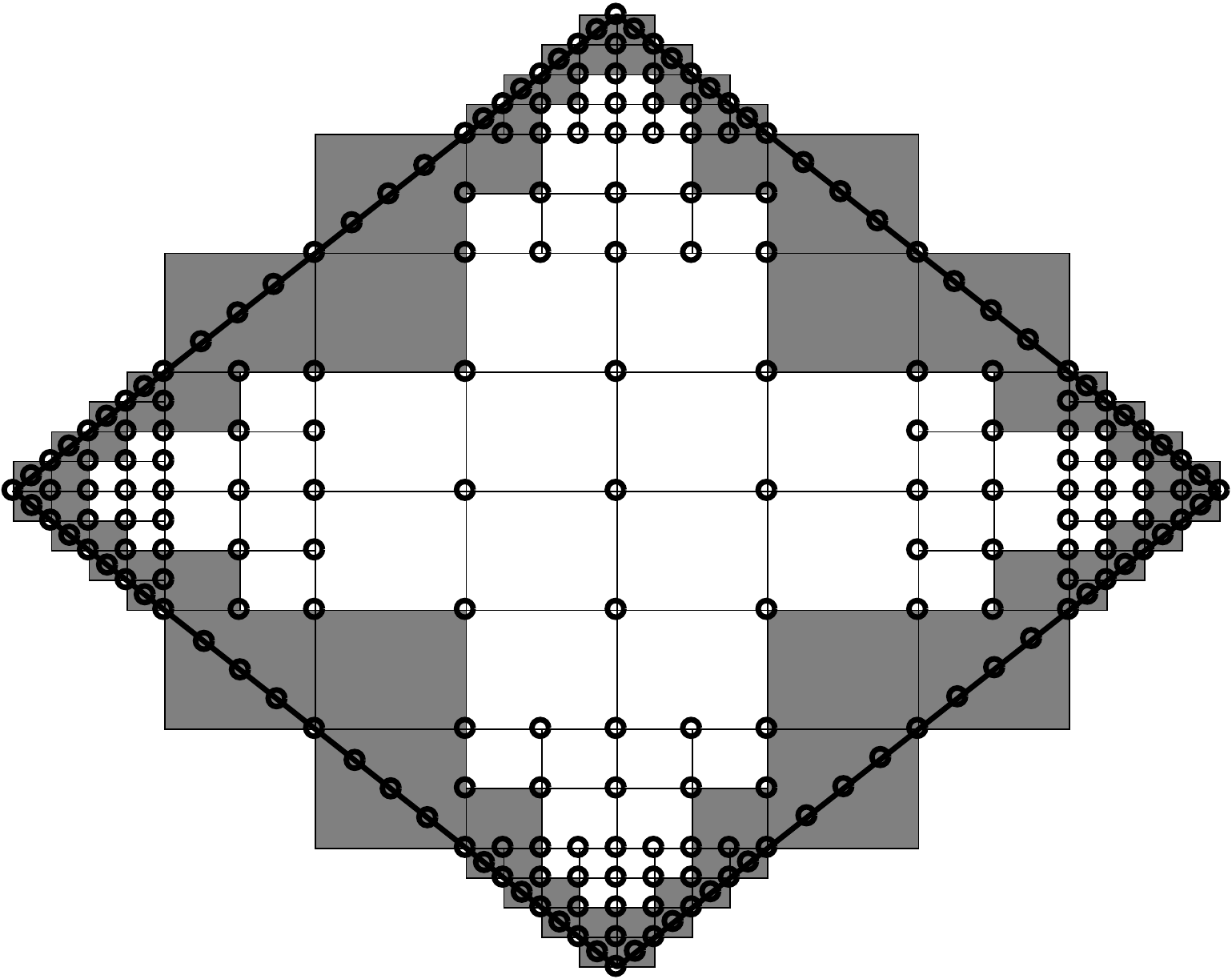}\label{fig:quad2}}
\caption{Examples of quadtree meshes. White squares are inside the domain, while gray squares intersect the boundary \cite{HS_Quadtree}.}\label{fig:quadtree}
\end{figure}

We require three parameters to describe the refinement of the mesh: the global resolution $h$, the boundary resolution $h_B$, and the gap $\delta$ between interior and boundary points.
\bq\label{eq:resolution} h = \sup\limits_{y\in X}\min\limits_{x\in\G^h}\abs{x-y}, \eq
\bq\label{eq:bdyRes} h_B = \sup\limits_{y\in\partial X}\min\limits_{x\in\G^h\cap X}\abs{x-y}, \eq
\bq\label{eq:delta} \delta = \min\limits_{x\in\G^h\cap X, y \in G^h\cap\partial X}\abs{x-y}. \eq
In order to construct consistent numerical methods, we require that $h_B = o(h)$ and $\delta = \bO(h)$ as $h\to0$.  This is easily accomplished as described in~\cite{HS_Quadtree}.

We will also associate to the mesh a directional resolution $d\theta$ and a search radius $r$, whose roles will become clear in the remainder of this section.  We choose
\bq\label{eq:dtheta}  d\theta = \bO(\sqrt{h}),\eq
\bq\label{eq:r} r = \bO(\sqrt{h}).\eq

\subsection{Approximation of second-order terms}
We now utilize the approach of~\cite{FroeseMeshfreeEigs,HS_Quadtree} to describe consistent, monotone approximations of the eigenvalues of the Hessian.  

We begin by noting that the second-order operators appearing in~\eqref{eq:modifiedPDE} all involve the eigenvalues of the Hessian matrix.  In two dimensions, these can be characterized in terms of the maximal and minimal second directional derivatives
\bq\label{eq:eigs}
\begin{aligned}
\lambda_1(D^2u) &= \min\limits_{\abs{\nu}=1} \frac{\partial^2u}{\partial\nu^2}\\
\lambda_2(D^2u) &= \max\limits_{\abs{\nu}=1} \frac{\partial^2u}{\partial\nu^2}.
\end{aligned}
\eq 

We begin by considering the approximation of the second directional derivative at a point $x_0$ along a generic direction $\nu\in\R^2$.  We first seek out candidate neighbors to use in discretizing this operator.  To begin, we consider all grid points within our search radius $r$.

Neighboring grid points can be written in polar coordinates $(\rho,\phi)$ with respect to the axes defined by the lines $x_0 + t\nu$, $x_0 + t\nu^\perp$.  We seek one neighboring discretization point in each quadrant described by these axes, with each neighbor aligning as closely as possible with the line $x_0 + t\nu$.  That is, we select the neighbors
\[ x_j \in \argmin\left\{{\sin^2\phi} \mid (\rho,\phi)\in\G^h\cap B(x_0,r) \text{ is in the $j$th quadrant}\right\}\]
for $j = 1, \ldots, 4$.  See Figure~\ref{fig:intNeighbours}.
Because of the ``wide-stencil'' nature of these approximations (since the search radius $r \gg h$), care must be taken near the boundary.  Our requirement that the boundary be sufficiently highly resolved ($h_B \ll h$) ensures that consistency can be maintained even at points $x_0$ close to the boundary.

\begin{figure}[ht]
\centering
\includegraphics[width=.65\textwidth]{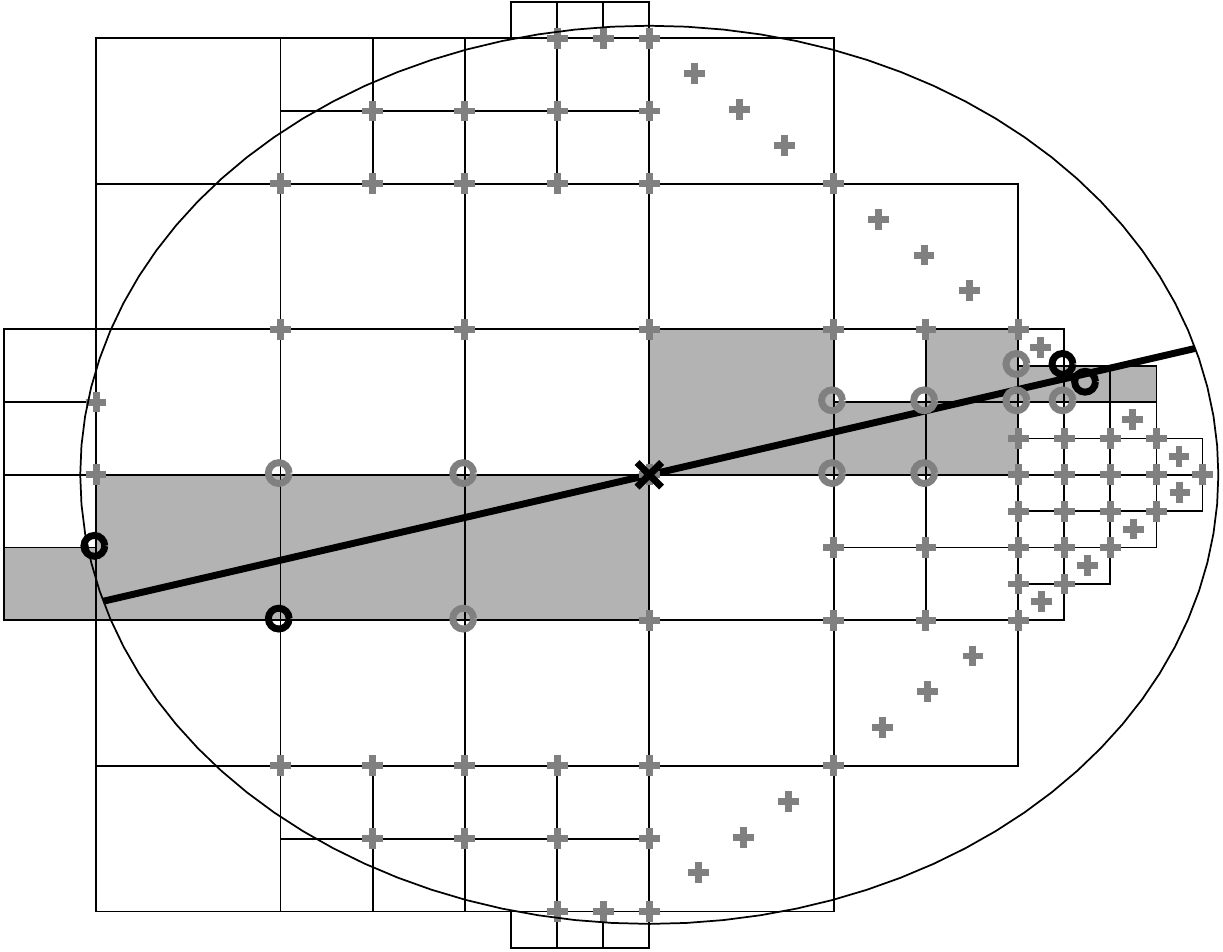}
\caption{Potential neighbors are circled in gray. Examples of selected neighbors are circled in black~\cite{HS_Quadtree}.}\label{fig:intNeighbours}
\end{figure}

We now seek an approximation of $\dfrac{\partial^2u(x_0)}{\partial\nu^2}$ of the form
\bq\label{eq:d2nu} \Dt_{\nu\nu}^hu(x_0) = \sum\limits_{j=1}^4 a_j(u(x_j)-u(x_0)). \eq
Via Taylor expansion of $u(x_j)$ (see~\cite{FroeseMeshfreeEigs}), we can verify that a consistent, (negative) monotone approximation is obtained with the coefficients
\[\begin{split}
a_1 &= \frac{2S_4(C_3S_2-C_2S_3)}{(C_3S_2-C_2S_3)(C_1^2S_4-C_4^2S_1)-(C_1S_4-C_4S_1)(C_3^2S_2-C_2^2S_3)}\\
a_2 &= \frac{2S_3(C_1S_4-C_4S_1)}{(C_3S_2-C_2S_3)(C_1^2S_4-C_4^2S_1)-(C_1S_4-C_4S_1)(C_3^2S_2-C_2^2S_3)}\\
a_3 &= \frac{-2S_2(C_1S_4-C_4S_1)}{(C_3S_2-C_2S_3)(C_1^2S_4-C_4^2S_1)-(C_1S_4-C_4S_1)(C_3^2S_2-C_2^2S_3)}\\
a_4 &= \frac{-2S_1(C_3S_2-C_2S_3)}{(C_3S_2-C_2S_3)(C_1^2S_4-C_4^2S_1)-(C_1S_4-C_4S_1)(C_3^2S_2-C_2^2S_3)}.
\end{split}
\]
where we use the polar coordinate characterization of the neighbors to define
\[ S_j =  \rho_j\sin\phi_j, \quad C_j = \rho_j\cos\phi_j.\]

We now want to use this to build up an approximation scheme for the eigenvalues of the Hessian via the characterization in~\eqref{eq:eigs}.  At the discrete level, instead of considering all possible second directional derivatives, we will consider a finite subset along the subset of unit directions
\bq\label{eq:directions}
V^h = \left\{(\cos(jd\theta),\sin(jd\theta)) \mid j = 1, \ldots, \floor{\dfrac{\pi}{d\theta}}\right\}.
\eq
Notice that $d\theta$, introduced in~\eqref{eq:dtheta}, now clearly describes the angular resolution of this subset of unit vectors.

We now approximate our second order operators by
\bq\label{eq:discEig}
\begin{aligned}
\lambda_1^h(x,u(x)-u(\cdot)) &= \min\limits_{\nu\in V^h} \Dt_{\nu\nu}u(x) \\
\lambda_2^h(x,u(x)-u(\cdot)) &= \max\limits_{\nu\in V^h} \Dt_{\nu\nu}u(x) \\
F^h(x,u(x)-u(\cdot)) &= -\arctan(\lambda_1^h(x,u(x)-u(\cdot))) - \arctan(\lambda_2^h(x,u(x)-u(\cdot)))\\
L^h(x,u(x)-u(\cdot)) &= -\lambda_1^h(x,u(x)-u(\cdot)).
\end{aligned}
\eq

This immediately leads to consistent, monotone approximations of our second-order operators since the underlying schemes for $\Dt_{\nu\nu}u$ are (negative) monotone.
\begin{lemma}[Second order approximations]\label{lem:monEig}
The approximation schemes $F^h$ and $L^h$ are consistent, monotone approximations of $F(D^2u)$ and $-\lambda_1(D^2u)$ respectively.
\end{lemma}

\subsection{Approximation of first-order terms}
Monotone approximations of the first-order terms in~\eqref{eq:modifiedPDE} have been more well-studied; of particular note are upwind~\cite{Qian_sweeping,Zhao_sweeping} and Lax-Friedrichs schemes~\cite{Kao_LF}.  Here we briefly review one choice of discretization for the first-order terms in the interior of the domain, which is fairly easily extended onto the non-uniform grids we are considering.

Several options are available for the Hamilton-Jacobi operator $H(\nabla u)$.  A simple choice is a modification of a standard Lax-Friedrichs scheme.

To construct this, we need to first describe generalizations of standard first-order differencing operators. Let $x_0\in\G^h$ and consider a differencing operator to approximate a partial derivative in the coordinate direction $\nu = (1,0)$.  Let $x_1, x_2, x_3, x_4$ be the neighbors used in the approximation of $\Dt_{\nu\nu}u(x_0)$~\eqref{eq:d2nu}.  We can use these same neighbors to generate a consistent forward difference type approximation of the first derivative $\dfrac{\partial u(x_0)}{\partial\nu}$ as
\bq\label{eq:d1nu}
\Dt_{(1,0)}^+u(x_0) = b_1(u(x_1)-u(x_0)) + b_4(u(x_4)-u(x_0)).
\eq
If either $x_1-x_0$ or $x_4-x_0$ is parallel to the direction $\nu = (1,0)$, we can choose a standard forward difference.  Otherwise, let $h_i = (x_i-x_0)\cdot(1,0)$ and $k_i = (x_i-x_0)\cdot(0,1)$ be the horizontal and vertical displacements.  Then a consistent scheme is given by
\[ b_1 = -\frac{k_4}{k_1h_4-h_1k_4}, \quad b_4 = \frac{k_1}{k_1h_4-h_1k_4}. \]
We can similarly define the differencing operators $\Dt_{(1,0)}^-$, $\Dt_{(0,1)}^+$, and $\Dt_{(0,1)}^-$.  We can also construct the centered type approximations
\bq\label{eq:d1nucent}
\Dt_\nu u(x) = \frac{1}{2}\left(\Dt_\nu^+ + \Dt_\nu^-\right)u(x).
\eq

We note that the signed distance function $H$ has Lipschitz constant one.  Then a consistent, monotone approximation at interior points $x\in X$ is given by
\bq\label{eq:LF}
H^h(x,u(x)-u(\cdot)) = H(\Dt_{(1,0)}u(x),\Dt_{(0,1)}u(x)) - \epsilon(h)\left(\Dt_{(1,0),(1,0)}u(x)+\Dt_{(0,1),(0,1)}u(x)\right)
\eq
where
\[ \epsilon(h) = \max\left\{\frac{\abs{b_j}}{a_j} \mid j = 1, \ldots, 4\right\}. \]
Note that this is a natural generalization of the standard Lax-Friedrichs approximation to our augmented piecewise Cartesian grids.

Finally, we need to approximate the Eikonal term $\abs{\nabla u}$. Again, many options are possible.  A slightly non-standard choice that is convenient for the convergence analysis involves characterizing this as the maximum possible first directional derivative,
\[ \abs{\nabla u(x)} = \max\limits_{\abs{\nu}=1}\frac{\partial u(x)}{\partial\nu}. \]
Then a simple choice of discretization involves looking at all possible directions that can be approximated exactly within our search radius $r$.
\bq\label{eq:eik}
E^h(x,u(x)-u(\cdot)) = \max\left\{\frac{u(x)-u(y)}{\abs{x-y}} \mid y\in\G^h\cap B(x,r)\right\}.
\eq

\subsection{Boundary conditions}
Finally, we need to discretize the Hamilton-Jacobi operator $H(\nabla u)$ at points on the boundary.  Using the representation~\eqref{eq:HJLF} and the angular discretization~\eqref{eq:directions}, we would like to express this as
\bq\label{eq:bcapprox}
H^h(x,u(x)-u(\cdot)) = \sup\limits_{n\in V^h, n\cdot n_x > 0}\left\{\Dt_n u(x) - H^*(n)\right\}
\eq
where $\Dt_n u$ is a monotone approximation of the directional derivative of $u$ in the direction $n$.

To accomplish this at a point $x_0\in\partial X$, we need to identify points $x_1, x_2\in\G^h$ such that for small $t>0$, the line segment $x_0 - n t$ is contained in the convex hull of $x_0, x_1, x_2$ (which is a triangle).  Given the structure of our mesh, this is easily accomplished for neighbors satisfying
\[ \abs{x_1-x_0}, \abs{x_2-x_0} \leq \bO(h). \]
See Figure~\ref{fig:BNfig} for a visual of this selection.  

Using these neighboring points, a consistent and monotone approximation for $\Dt_{n}u(x_0)$ can be built in exactly the same way as the forward differencing operator $\Dt_\nu^+$ was constructed in~\eqref{eq:d1nu}.  That is, we let $n^\perp$ be a unit vector orthogonal to $n$ and define
\[ h_i = (x_i-x_0)\cdot n, \quad k_i = (x_i-x_0)\cdot n^\perp, \]
noting that $h_i\leq 0$ for $i=1,2$ and $k_1k_2 \leq 0$.  Then a consistent, monotone approximation is given by
\bq\label{eq:bdyDirApprox}
\Dt_n u(x_0) = \frac{-k_2}{k_1h_2-h_1k_2}(u(x_1)-u(x_0)) + \frac{k_1}{k_1h_2-h_1k_2}(u(x_2)-u(x_0)).
\eq

\begin{figure}[ht]
\centering
\includegraphics[width=0.45\textwidth]{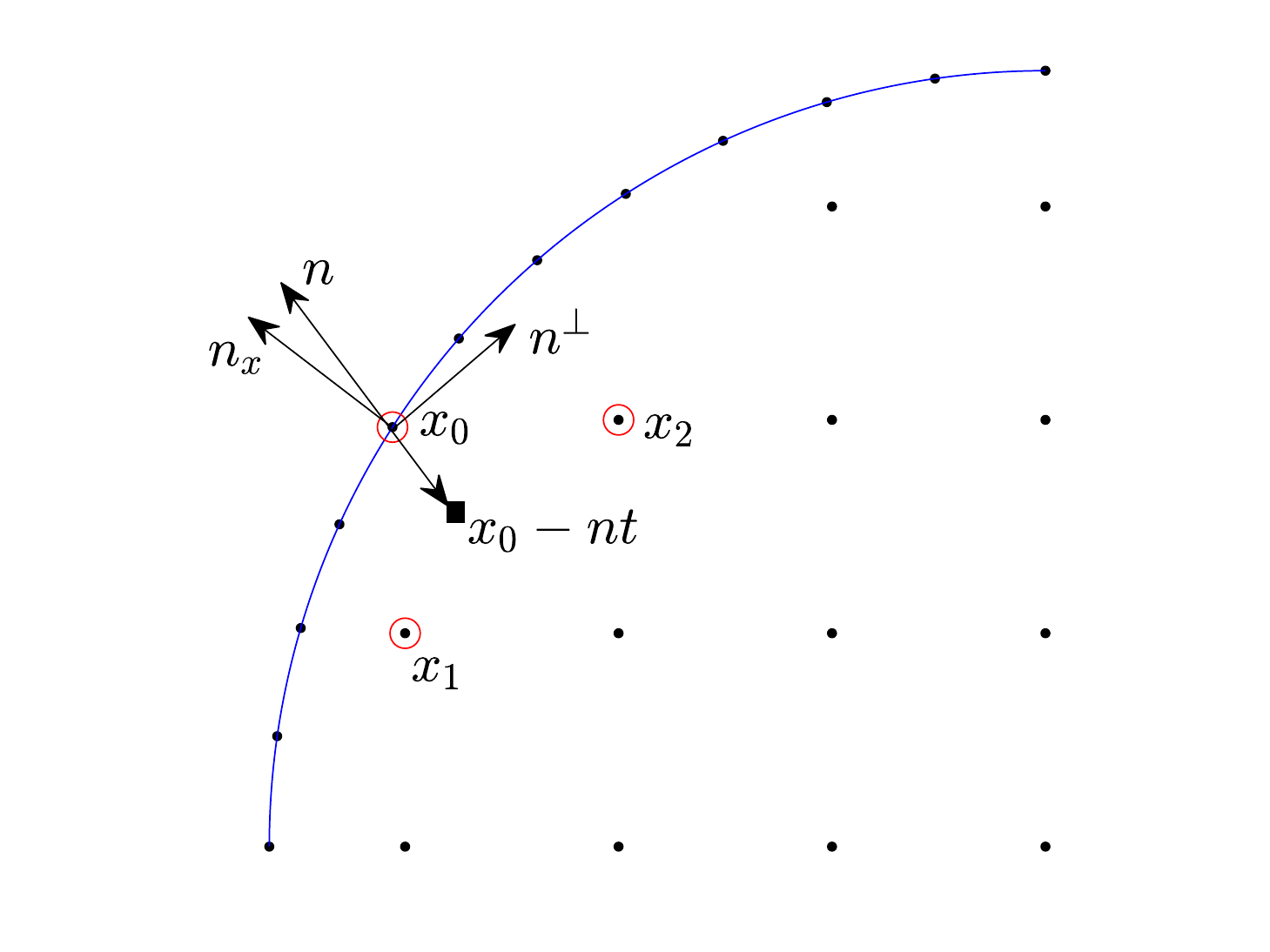}
\includegraphics[width=0.45\textwidth]{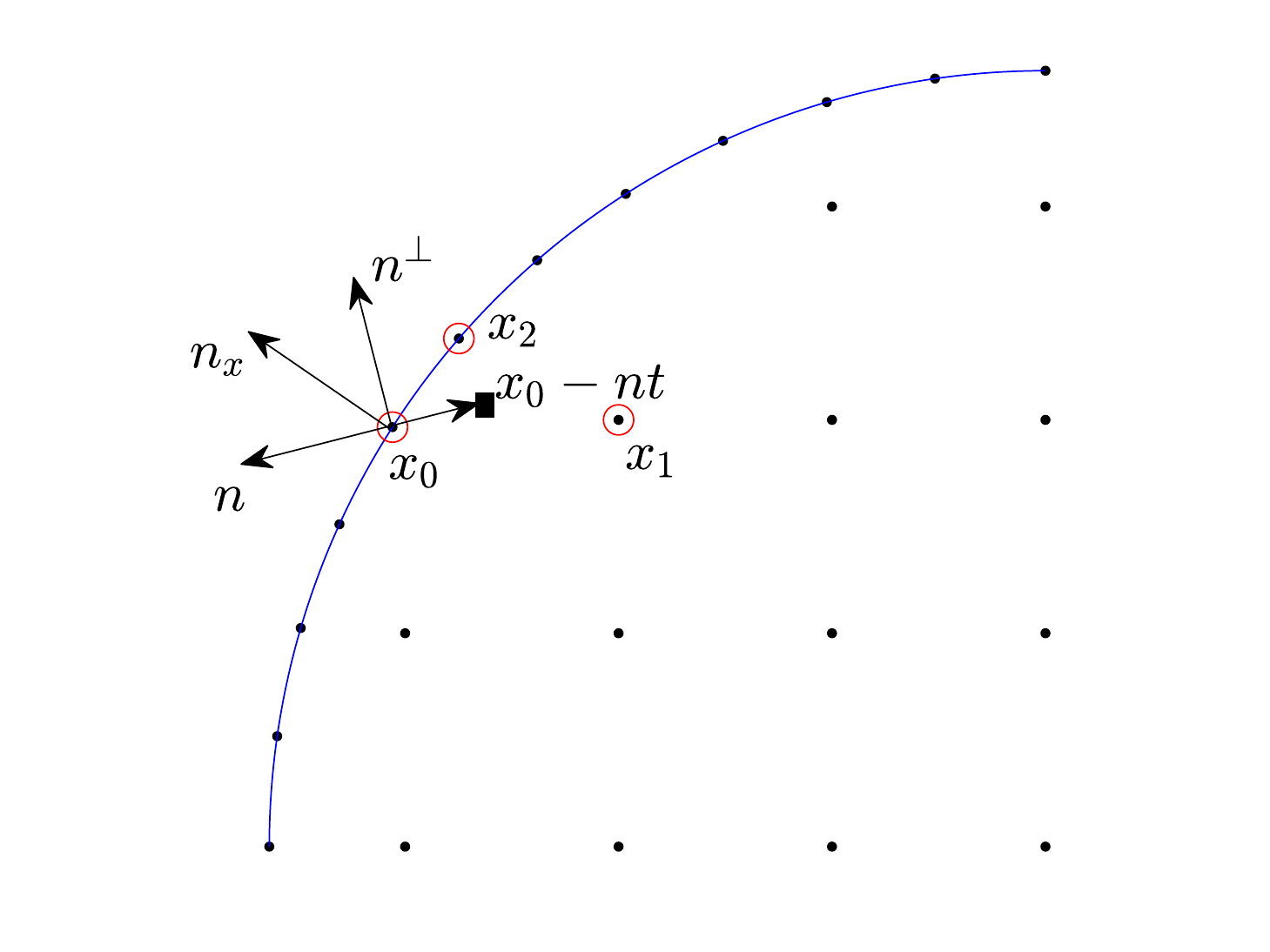}
\caption{Examples of neighbors $x_1, x_2$ needed to construct a monotone approximation of the directional derivative in the direction $n$ at the boundary point $x_0$.}
\label{fig:BNfig}
\end{figure}

We remark that the discrete operator $E^h$ at the boundary is unchanged from the form used on interior points; see~\eqref{eq:eik}.

\section{Convergence Analysis}\label{sec:analysis}
We are now prepared to state our convergence results for the numerical scheme presented in the previous section.  We separate this into two results: convergence of the eigenvalue $c^h$ and convergence of the grid function $u^h$.

\begin{theorem}[Convergence of the eigenvalue]\label{thm:convEig}
Let $(\uex,\cex)$ be any solution of the eigenvalue problem~\eqref{eq:LagrangianPDE}-\eqref{eq:BVP2} and let $(v^h,c^h)$ be any solution of the scheme~\eqref{eq:disc1}.  Then $c^h$ converges to $\cex$ as $h\to0$.
\end{theorem}

\begin{theorem}[Convergence of the grid function]\label{thm:convSol}
Let $(\uex,\cex)$ be a solution of the eigenvalue problem~\eqref{eq:LagrangianPDE}-\eqref{eq:BVP2} satisfying $\uex(x_0) = 0$ and let $u^h$ be any solution of the scheme~\eqref{eq:disc2}-\eqref{eq:uh}.  Then $u^h$ converges uniformly to $\uex$ as $h\to0$.
\end{theorem}

We also remark that, while our focus here is the construction of minimal Lagrangian graphs, this analysis could be readily adapted to more general eigenvalue problems of the form~\eqref{eq:EigProblem}.

\subsection{Convergence of the Eigenvalue}
We begin by establishing convergence of the eigenvalue (Theorem~\ref{thm:convEig}).
The proof of this result will require several short lemmas.  In these we will use the shorthand notation
\begin{align*}
F_i^h[u] &= F^h(x_i,u(x_i)-u(\cdot))\\
H_i^h[u] &= H^h(x_i,u(x_i)-u(\cdot)).
\end{align*}
We also define the following objects relating to sub- and super-solutions of the schemes.
\bq\label{eq:setSubs}
U_c^h = \{u \mid F^h_i[u]+c \leq 0, x_i \in \G^h\cap X; H^h_i[u] < 0, x_i\in\G^h\cap\partial X
\}
\eq
\bq\label{eq:setSuper}
V_c^h = \{v \mid F^h_i[u]+c \geq 0, x_i \in \G^h\cap X; H^h_i[u] > 0, x_i\in\G^h\cap\partial X\}
\eq

We begin by establishing that these sets of sub(super)-solutions are non-empty for appropriate choices of $c$.
\begin{lemma}[existence of sub(super) solutions]\label{lem:subExist}
There exist $u_+^h\in V_{c_{\text{ex}}+\omega(h)}^h$, $u_-^+h\in U_{c_{\text{ex}}-\omega(h)}$ where $\omega(h)$ is proportional to the maximum consistency error $\tau(h)$ of the scheme.
\end{lemma}

\begin{proof}
We begin by letting $D(x)$ be the signed distance function to the boundary of the domain $\partial X$. Note that $D$ is smooth in a $\delta$-neighborhood of the boundary $\partial X$ for some $\delta>0$.  Next, we let $\phi$ be a smooth cut-off function satisfying
\[
\phi(x) = \begin{cases}1, & \text{dist}(x,\partial X)<\delta/2\\
0, & \text{dist}(x,\partial X)\geq\delta.
\end{cases}\]
We can then define a smooth function 
\[ w(x) = D(x)\phi(x). \]
Notice that on the boundary $\partial X$, this will satisfy
\[ \nabla w(x) = \nabla D(x) = n_x, \quad x \in \partial X. \]

We choose some $\epsilon > \dfrac{\max\{C_1,C_2\}}{\ell}\tau(h)$, where the constants $C_1, C_2$ will be fixed later, and define
\[ u_-^h =  u_{\text{ex}} - \epsilon w, \quad u_+^h = u_{\text{ex}} + \epsilon w.\]
We will show that for suitable choices of $C_1, C_2$ and $\omega(h) = \bO(\tau(h))$ we have $u_-^h \in U_{c_{\text{ex}}-\omega(h)}^h$.  The argument regarding $u_+^h$ is similar.

Note that from Lemma~\eqref{lem:HJOpHyperplane}, if $x\in\partial X$ then
\begin{align*} H(\nabla u_-^h(x)) &= \sup\limits_{n\cdot n_x>\ell}\{\nabla u_-^h(x)\cdot n - H^*(n)\} \\
  &\leq \sup\limits_{n\cdot n_x>\ell}\{\nabla \uex(x)-H^*(n)\} - \epsilon\ell\\
	&= H(\nabla \uex(x)) - \epsilon\ell\\
	&= -\epsilon\ell.
\end{align*}
By consistency we have
\[ H_i^h[u_-^h] \leq  H(\nabla u_-^h(x_i)) + C_1\tau(h) \leq -\epsilon\ell+C_1\tau(h)\]
for some $C_1>0$.
Since $\epsilon > \dfrac{C_1\tau(h)}{\ell}$, we obtain
\[ H^h[u_-^h] < 0. \]

Since our PDE operator $F$ is Lipschitz, we can also find some $L>0$ so that
\[ F(D^2u_-^h(x)) \leq F(D^2u_{\text{ex}}(x)) + L\epsilon = -c_{\text{ex}}+L\epsilon. \]
By consistency we again have
\[ F^h_i[u_-^h] \leq F(D^2u_-^h(x_i)) + C_2\tau(h) \leq -c_{\text{ex}} + L\epsilon + C_2\tau(h) \]
Since $\epsilon > \dfrac{C_2\tau(h)}{\ell}$ and defining $\omega(h) = L\epsilon(h)+\tau(h)$ we have
\[ F^h[u_-^h] + c_{\text{ex}}-\omega(h) \leq 0. \]
We conclude that $u_-^h \in U_{c_{\text{ex}}-\omega(h)}^h$.
\end{proof}

Now using the discrete comparison principle, we can begin to see how the sets of sub(super)-solutions are related to each other, which will lead ultimately to constraints on our numerically computed eigenvalue.
\begin{lemma}[Comparison of eigenvalues]\label{lem:eigCompare}
Suppose $u_1 \in U^h_{c_1}$ and $u_2\in V^h_{c_2}$. Then $c_1 \leq c_2$.
\end{lemma}

\begin{proof}
Suppose instead that $c_1>c_2$.  Note that for any constant $k$ we also have $u_1 + k \in U^h_{c_1}$.  Thus, we can assume that $u_1 > u_2$.  Now we estimate
\[ F^h_i[u_1] + c_2 < F^h_i[u_1] + c_1 \leq 0 \leq F^h_i[u_2] + c_2, \quad x_i \in \G^h\cap X \]
and
\[ H^h_i[u_1] <  0 < H^h_i[u_2], \quad x_i \in \G^h\cap\partial X.\]
By the discrete comparison principle (Lemma~\ref{lem:discreteComp}) we have $u_1 \leq u_2$, a contradiction.
\end{proof}

With these lemmas in place, we can now prove convergence of the numerically computed eigenvalue.

\begin{proof}[Proof of Theorem~\ref{thm:convEig}]
Recall that
\[ F_i^h[v^h] + c^h = 0, \quad x_i \in \G^h\cap X\]
and
\[ H_i^h[v^h] = 0, \quad x_i \in\G^h\cap\partial X. \]
Following Lemmas~\ref{lem:subExist}-\ref{lem:eigCompare} we conclude that
\[ c_{\text{ex}}-\omega(h) \leq c^h \leq c_{\text{ex}}+\omega(h). \qedhere \]
\end{proof}

Having proved the convergence of the eigenvalue $c^h \to c_{\text{ex}}$, this reduces our task from the convergence of an eigenvalue problem to the convergence of a fully nonlinear elliptic PDE.

\subsection{Convergence of the grid function}
We now turn our attention to the convergence of the approximation $u^h$ to the solution $\uex$ (Theorem~\ref{thm:convSol}).

In order to prove this theorem, we first need to construct a piecewise linear extension $\ulin^h$ of the grid function $u^h$.

Let $T^h$ be a triangulation of $\G^h$.  In particular, given the structure of our balanced quadtree mesh (augmented on the boundary), we can construct such a triangulation such that the maximal angle of any triangle is bounded uniformly away from $\pi$.
\begin{definition}[Structure of triangulation]\label{def:tri}
Define $T^h$ to be a triangulation of $\G^h$ satisfying the following properties:
\begin{enumerate}
\item[(a)] There exists some $M<1$ (independent of $h$) such that if $t \in T^h$ has the interior angles $\theta_1 \leq \theta_2 \leq \theta_3$ then 
\bq\label{eq:angles}
\abs{\cos\theta_3} \leq M.
\eq
\item[(b)] If $t\in T^h$, then at most two nodes of $t$ are contained on the boundary $\partial X$.
\item[(c)] If $t\in T^h$, then the diameter of $t$ is bounded by $2h$.
\end{enumerate}
\end{definition}

We note that since $\G^h \subset \bar{X}$, we need to extend triangles that intersect the boundary in order to obtain a decomposition $\tilde{T}^h$ that fully covers the domain.  To do this, we define the regions $\tilde{t}_i$ as follows:
\begin{definition}[Extension of triangulation]\label{def:regions}
Let $t \in T^h$, with the nodes $x_0,x_1,x_2$.  Then we define the corresponding region $\tilde{t}\in \tilde{T}^h$ as follows:
\begin{enumerate}
\item[(a)] If at least two nodes of $t$ are in $X$, set
\[ \tilde{t} = t. \]
\item[(b)] If two nodes $x_1, x_2 \in \partial X$, set
\[ \tilde{t} = \text{Conv}\{x_0, x_0+2(x_1-x_0),x_0+2(x_2-x_0)\} \cap \bar{X}. \]
\end{enumerate}
\end{definition}
We remark that 
\[ \bigcup_{\tilde{t}\in\tilde{T}} \tilde{t} = \bar{X}. \]

Now we are able to define a continuous piecewise linear extension.
\begin{definition}[Extension of grid function]\label{def:pwlin}
Define the unique continuous piecewise linear function $\ulin^h$ satisfying:
\begin{enumerate}
\item[(a)] $\ulin^h(x) = u^h(x)$ for all $x\in\G^h$.
\item[(b)] $\ulin^h(x)$ is a linear function on each region $\tilde{t}\in\tilde{T}$.
\end{enumerate}
\end{definition}
We remark that $\ulin^h$ will also satisfy the approximation scheme~\eqref{eq:uh}-\eqref{eq:uhBdy}.

An important element to our convergence proof will be to establish uniform Lipschitz bounds on the approximations~$\ulin^h$.
\begin{lemma}[Lipschitz bounds]\label{lem:lipschitz}
There exists a constant $L>0$ such that the Lipschitz constant of $\ulin^h$ is bounded by $L$ for all sufficiently small $h>0$.
\end{lemma}
\begin{proof}
We begin by considering the function $\ulin^h$ restricted to some fixed region $\tilde{t} \in \tilde{T}^h$.  Let $x_0, x_1, x_2$ be the nodes of $\tilde{t}$. Without loss of generality, we can assume that the maximal interior angle $\theta$ of $\tilde{t}$ occurs at the node $x_0$.

Now we know that $x_i\in\G^h$ for $i=0,1,2$.  Since $\ulin^h$ satisfies the scheme~\eqref{eq:uhBdy}, we know that
\[ E^h(x_i,\ulin^h(x_i)-\ulin^h(\cdot)) - R \leq 0. \]
From the definition of $E^h$~\eqref{eq:eik}, we can conclude that
\[ \ulin^h(x_i) \leq \ulin^h(y) + R\abs{x_i-y} \]
for every $y\in \G^h\cap B(x_i,r)$.  In particular, this holds for $y = x_0, x_1, x_2$ since the diameter of $\tilde{t}$ is bounded by $2h < r = \bO(\sqrt{h})$ for small enough $h>0$.  Thus, we obtain the discrete Lipschitz bounds
\[ \abs{\ulin^h(x_i)-\ulin^h(x_j)} \leq R\abs{x_i-x_j}, \quad i,j\in\{0,1,2\}. \]

We now use this to bound the gradient of $\ulin^h$ over the region $\tilde{t}$.  Notice that for $x\in\tilde{t}$ we can write
\[ \ulin^h(x) = \ulin^h(x_0) + p\cdot(x-x_0)\]
where $p = \nabla\ulin^h(x)$ is constant over this region.  Since $x_1-x_0$ and $x_2-x_0$ span $\R^2$, we can find constants $a_1, a_2\in\R$ such that
\[ p = a_1(x_1-x_0) + a_2(x_2-x_0). \]

Now we use our discrete Lipschitz bounds to compute
\[ R\abs{x_1-x_0} \geq \abs{u_1-u_0} = \abs{x_1-x_0}\abs{a_1\abs{x_1-x_0} + a_2\abs{x_2-x_0}\cos\theta}. \]
Simplifying and applying the bound on the maximal angle (Definition~\ref{def:tri}) we obtain
\[ R \geq \abs{a_1}\abs{x_1-x_0} - M \abs{a_2} \abs{x_2-x_0}. \]
Similarly,
\[ R \geq \abs{a_2}\abs{x_2-x_0} - M\abs{a_1}\abs{x_1-x_0}. \]
Combining the two above expressions, we find that
\[ R(M+1) \geq (1-M^2)\abs{a_2}\abs{x_2-x_0} \]
and thus
\[ \abs{a_2} \leq \frac{R}{(1-M)\abs{x_2-x_0}}. \]
An equivalent bound is available for $\abs{a_1}$.

Now we can bound $p = \nabla\ulin^h(x)$ over the region $\tilde{t}$ by
\[ \abs{p} \leq \abs{a_1}\abs{x_1-x_0} + \abs{a_2}\abs{x_2-x_0} \leq \frac{2R}{1-M} \equiv L. \]

Since $\ulin^h$ is piecewise linear, its Lipschitz constant will be bounded by the maximum Lipschitz constant over each region $\tilde{t}\in\tilde{T}^h$, which is given by $L$.
\end{proof}

An immediate consequence of this is uniform bounds for $\ulin^h$.
\begin{lemma}\label{lem:uniformbounduh}
There exists a constant $C>0$ such that $\|\ulin^h\|_\infty \leq C$ for all sufficiently small $h>0$.
\end{lemma}
\begin{proof}
Since $\ulin^h(x_0^h) = 0$ and $\ulin^h$ has a bounded Lipschitz constant (Lemma~\ref{lem:lipschitz}), we have that
$$
\begin{aligned}
\abs{\ulin^h(x)}  &= \abs{\ulin^h(x)-\ulin^h(x_0)}\\
 &\leq L \abs{x-x_0} \\
 &\leq L \: \text{diam}(X)
\end{aligned}
$$
for every $x\in\bar{X}$.
\end{proof}

Next we adapt the usual Barles-Souganidis convergence proof~\cite{BSnum} to begin to show how we can obtain viscosity solutions to~\eqref{eq:modifiedPDE} from our approximation scheme~\eqref{eq:disc2}.
\begin{lemma}\label{lem:subandsuper} Let $h_n$ be any sequence such that $h_n\to0$ and $\ulin^{h_n}$ converges uniformly to a continuous function $v$. Then $v$ is a viscosity solution of~\eqref{eq:modifiedPDE}.\end{lemma}
\begin{proof}
We first demonstrate that $v$ is a viscosity subsolution.

Consider any $x_0\in X$ and $\phi\in C^2$ such that $v-\phi$ has a strict local maximum at $x_0$ with $v(x_0) = \phi(x_0)$.  
Define by $z_n\in\G^h$ a maximizer of $\ulin^{h_n}-\phi$ over the grid,
\[\ulin^{h_n}(z_n)-\phi(z_n) \geq \ulin^{h_n}(x)-\phi(x), \quad x \in \G^{h_n}.\]
Because $\ulin^h$ and the limit function $v$ are uniformly Lipschitz continuous, strict maxima are stable and we have
\[ z_n\to x_0,  \quad \ulin^{h_n}(z_n)\to v(x_0). \]


From the definition of $z_n$ as a maximizer of $\ulin^{h_n}-\phi$, we also observe that
\[ \ulin^{h_n}(z_n)-\ulin^{h_n}(\cdot) \geq \phi(z_n)-\phi(\cdot). \]

Let $G(\nabla u(x),D^2u(x))$ denote the PDE operator~\eqref{eq:modifiedPDE} and $G^h(x,u(x)-u(\cdot))$ the scheme~\eqref{eq:uhInt} at interior points $x\in\G^h\cap X$.  Since $\ulin^{h_n}$ is a solution of the scheme, we can use monotonicity to calculate
\[ 0 = G^{h_n}(z_n,\ulin^{h_n}(z_n)-\ulin^{h_n}(\cdot)) \geq G^{h_n}(z_n,\phi(z_n)-\phi(\cdot)). \]

As the scheme is consistent, we conclude that
\[ 0 \geq \lim\limits_{n\to\infty}G^{h_n}(z_n,\phi(z_n)-\phi(\cdot)) = G(x_0,\nabla\phi(x_0),D^2\phi(x_0)). \]
Thus $v$ is a subsolution of~\eqref{eq:modifiedPDE}. 

An identical argument shows that $v$ is a supersolution and therefore a viscosity solution.
\end{proof}

With these lemmas in place, we can now complete the main convergence result.
\begin{proof}[Proof of Theorem~\ref{thm:convSol}]
Let $h_n$ be any sequence converging to $0$.  Since $\ulin^{h_n}$ is uniformly bounded and Lipschitz continuous (Lemmas~\ref{lem:lipschitz}-\ref{lem:uniformbounduh}), we can apply the Arzela-Ascoli theorem to obtain a subsequence $h_{n_k}$ such that
$u^{h_{n_k}} \to v$
uniformly for some continuous function $v$.

By Lemma~\ref{lem:subandsuper}, $v$ is a viscosity solution of~\eqref{eq:modifiedPDE} and therefore a classical solution of the eigenvalue problem~\eqref{eq:LagrangianPDE}-\eqref{eq:BVP2}.  Moreover, since convergence is uniform and $u^{h_{n_k}}$ continuous we have that
\[ v(x_0) = \lim\limits_{k\to\infty}\ulin^{h_{n_k}}(x_0^{h_{n_k}}) = 0. \]
Thus $v=\uex$ is the unique solution of~\eqref{eq:LagrangianPDE}-\eqref{eq:BVP2} satisfying $v(x_0) = 0$.

Since every sequence $\ulin^{h_n}$ has a subsequence converging to $\uex$, we conclude that $\ulin^h$ converges to $\uex$.
\end{proof}

\section{Computational Results}\label{sec:results}

We now present some numerical results to illustrate the effectiveness of our methods.  

These computations require solving the nonlinear algebraic system~\eqref{eq:disc1} \[G^h[v^h;c^h] = 0\] for the unknowns $v^h$ and $c^h$.  While the system is not differentiable, the non-smoothness occurs in a simple form through the max function. These systems can be solved using a nonsmooth version of Newton's method~\cite{Qi_nonsmoothNewton}, which involves the iteration
\[ \left(\begin{matrix} v_{k+1}\\c_{k+1}\end{matrix}\right) = \left(\begin{matrix} v_{k}\\c_{k}\end{matrix}\right) - V_k^{-1}G^h[v_k;c_k] \]
where $V_k\in\partial G^h[v_k;c_k]$ is an element of the generalized Jacobian of $G^h$.  Given the simple form of the non-smoothness in this problem, appropriate elements of the generalized Jacobian are easily computed via Danskin's Theorem~\cite{Bertsekas}.

We note that the solution $v^h$ we obtained here always satisfied~\eqref{eq:disc2} so we have never found it necessary to solve this second system in practice.

\subsection{Affine surface}
We begin with an example where the minimal Lagrangian surface $\nabla u$ is affine, which allows us to exactly determine the error in our computed results.

Let $B$ be the unit circle. Then the domain ellipse is given by $X = M_x B$ and the target skew ellipse is given by $Y = M_y B$, where 
$$M_x = \begin{bmatrix} 2 & 0\\
0 & 1
\end{bmatrix}$$
and 
$$M_y = \begin{bmatrix} 1.5 & .5\\
.5 & 2
\end{bmatrix}$$

 In $\R^2$ the optimal map can be found explicitly to be $$\nabla u(x) = M_y R_\theta M_{x}^{-1}x$$ where $R_\theta$ is the rotation matrix
$$R_\theta = \begin{bmatrix} \cos{(\theta)} & -\sin{(\theta)}\\
\sin{(\theta)} & \cos{(\theta)}
\end{bmatrix}$$
and the angle is given by 
$$\theta = \itan{(\text{trace}(M_{x}^{-1}M_{y}^{-1}J)/\text{trace}(M_{x}^{-1}M_{y}^{-1}))}$$
where $$J = R_{\pi/2} = \begin{bmatrix} 0 & -1\\
1 & 0
\end{bmatrix}$$
The map and the convergence data are in Figure~\ref{ESEmap} and Table~\ref{ESEdata}.  In this example, we actually observe linear convergence, which is higher than the formal discretization error of our method.

\begin{table}[ht]
\begin{center}
\begin{tabular}{cccc}
$h$          & $\|u^h-\uex\|_\infty$       & Ratio   & Observed Order \\
\hline
   $2.625\times 10^{-1}$&   $1.304 \times 10^{-1}$    &      &           \\
   $1.313\times 10^{-1}$ &  $5.703\times 10^{-2}$   &2.287   &      1.194 \\
   $6.563\times 10^{-2}$  & $2.691\times 10^{-2}$   &2.119     &    1.084 \\
   $3.281\times 10^{-2}$   &$1.423\times 10^{-2}$  & 1.891       &  0.919 \\
   $1.641\times 10^{-2}$  & $6.768\times 10^{-3}$ &  2.103       &  1.072
\end{tabular}
\end{center}
\caption{Error in mapping an ellipse to an ellipse.}
\label{ESEdata}
\end{table}

\begin{figure}[!h]
		\begin{center}
\includegraphics[width=5.715cm,height=54.5mm]{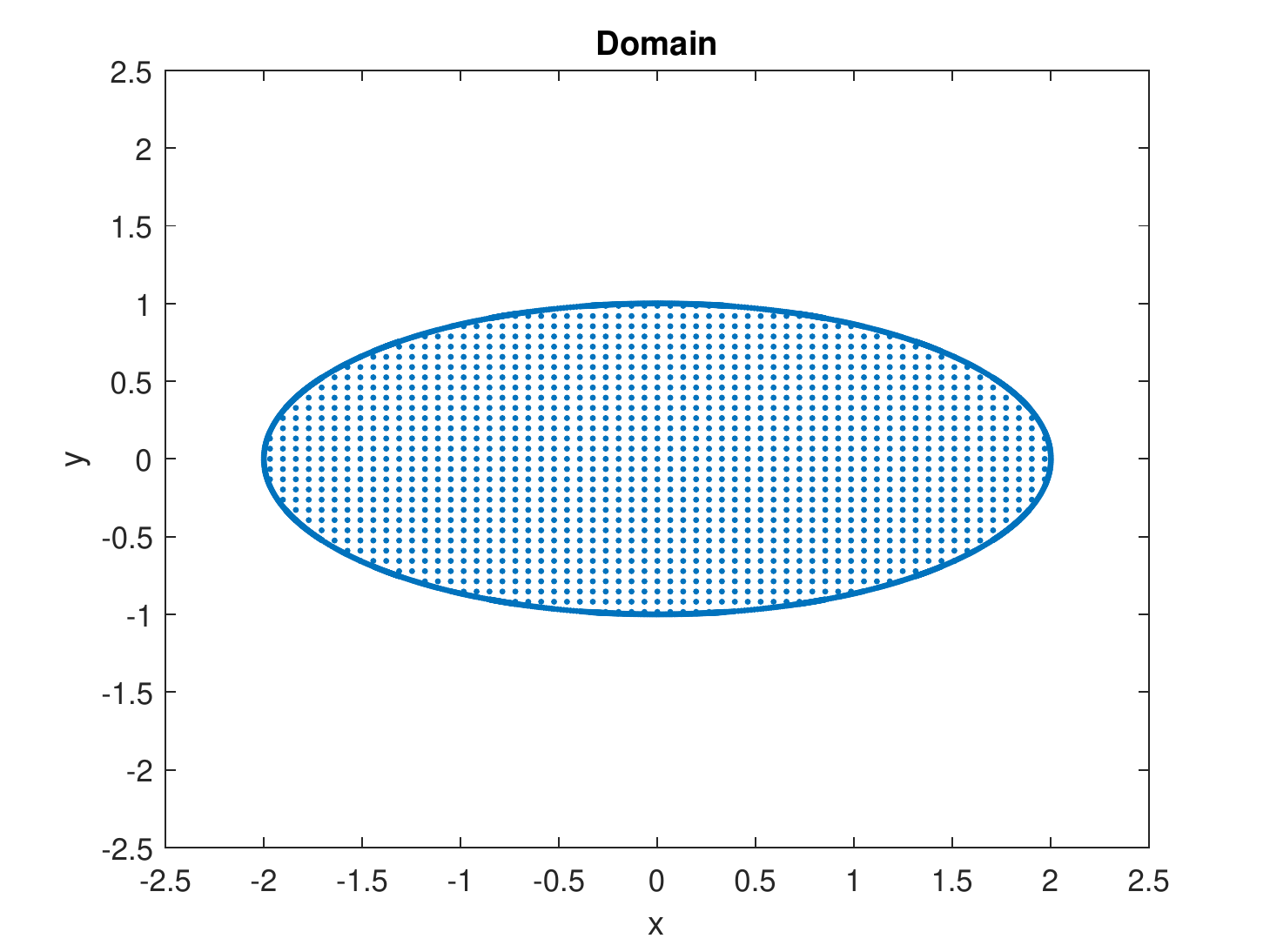}\includegraphics[width=5.715cm,height=54.5mm]{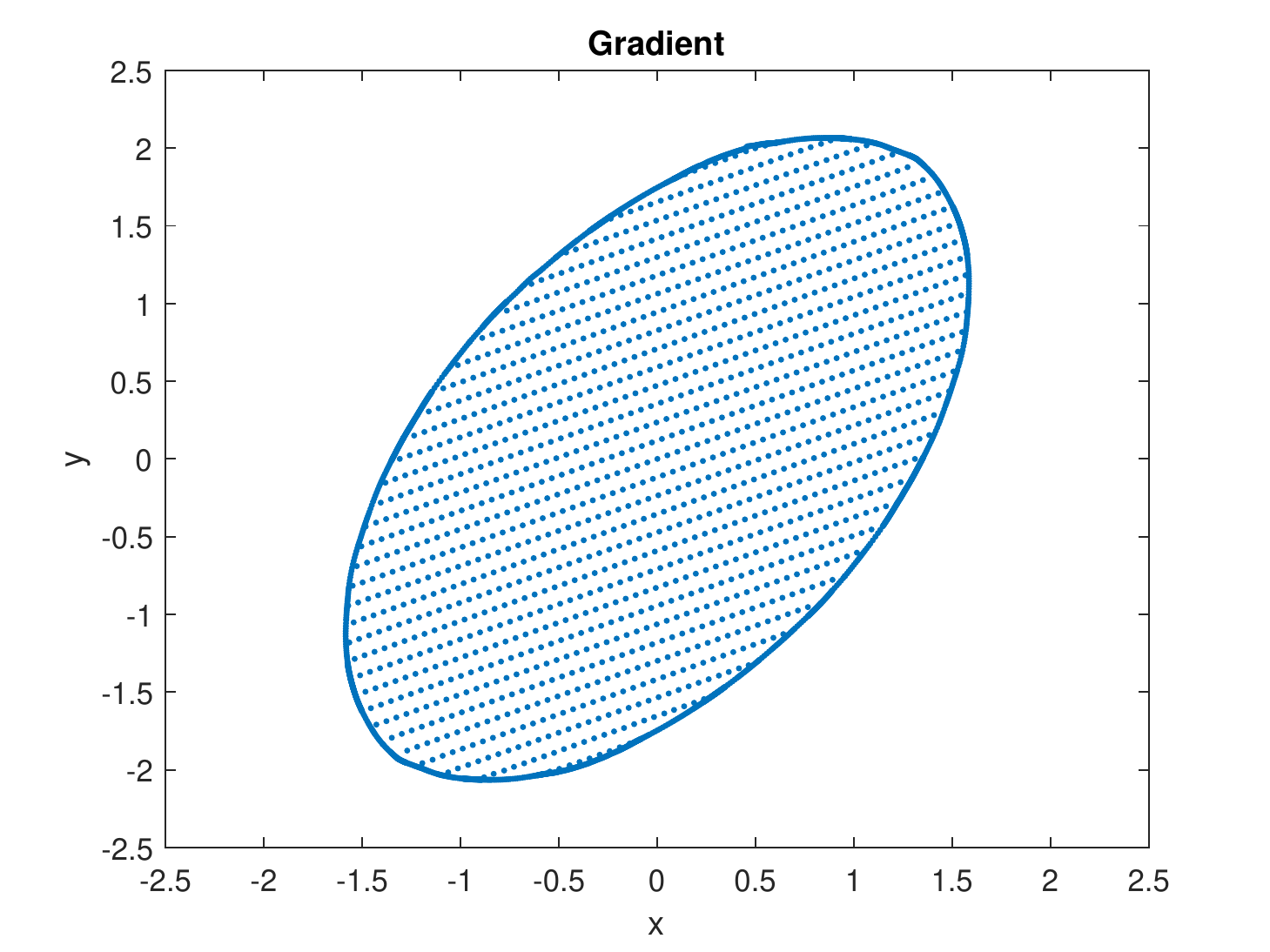}
		\end{center}
\caption{Domain and computed target ellipse.}
\label{ESEmap}
\end{figure}

\subsection{Varying boundary conditions}
For most examples, we do not have access to an exact solution.  Nonetheless, we can easily compute the solutions and visually determine if the computed mapping $\nabla u$ appears correct. In the following examples,  we take as our domain the square $X = (-1.1,1.1)\times(-1.1,1.1)$, which is not required to align well with any underlying Cartesian grid in order to challenge our numerical method. We consider the solution of~\eqref{eq:LagrangianPDE}-\eqref{eq:BVP2} for a variety of convex (though not necessarily uniformly convex) target sets $Y$. These include a bowl shape, an ice cream cone, a pentagon, and a circle. The computed maps are pictured in Figure~\ref{fig:shapes} and do effectively recover the required geometries.

\begin{figure}[ht]
\centering
\subfigure[]{
\includegraphics[width=5.715cm,height=54.5mm]{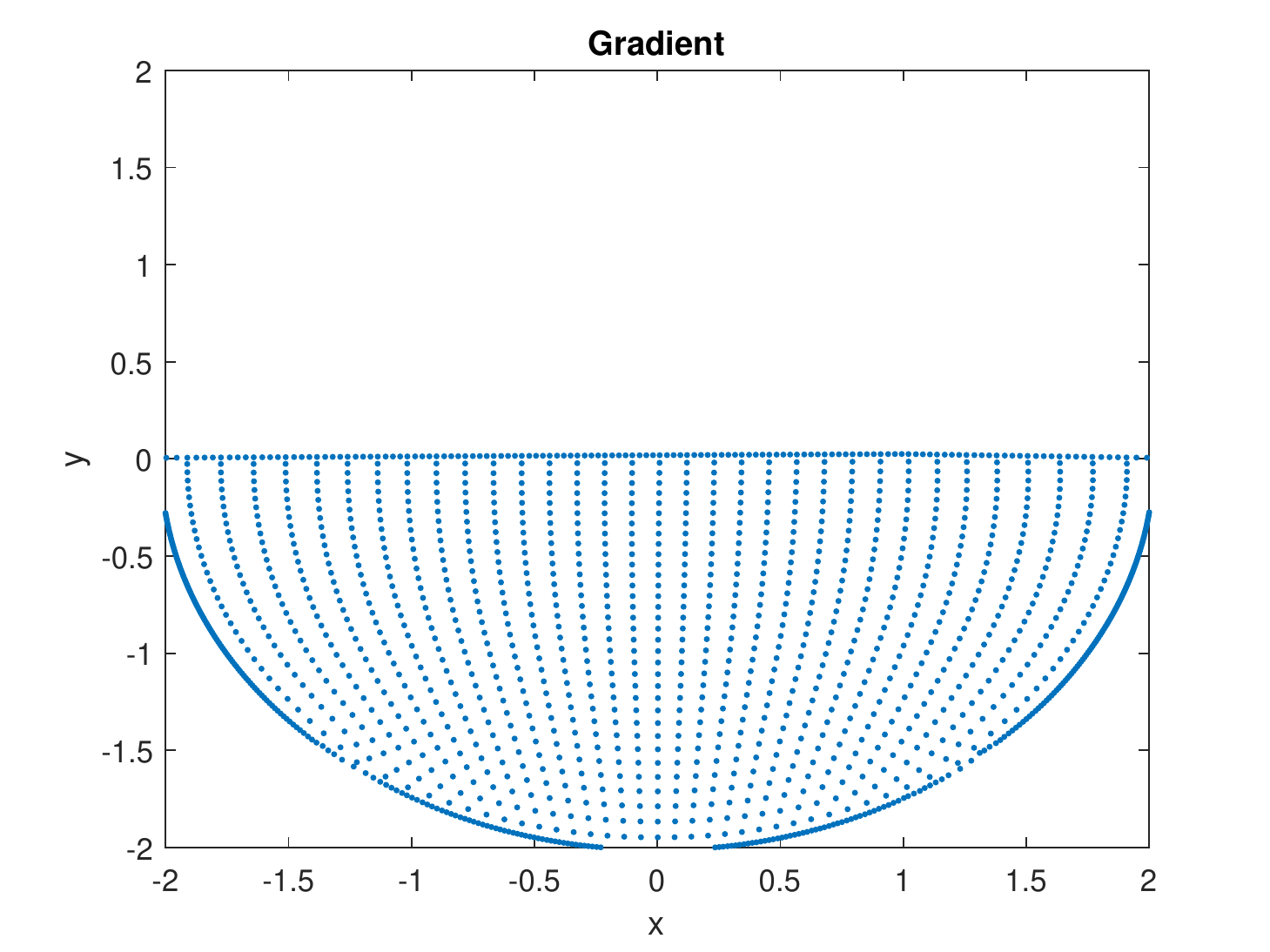}\label{fig:quad1}}
\subfigure[]{
\includegraphics[width=5.715cm,height=54.5mm]{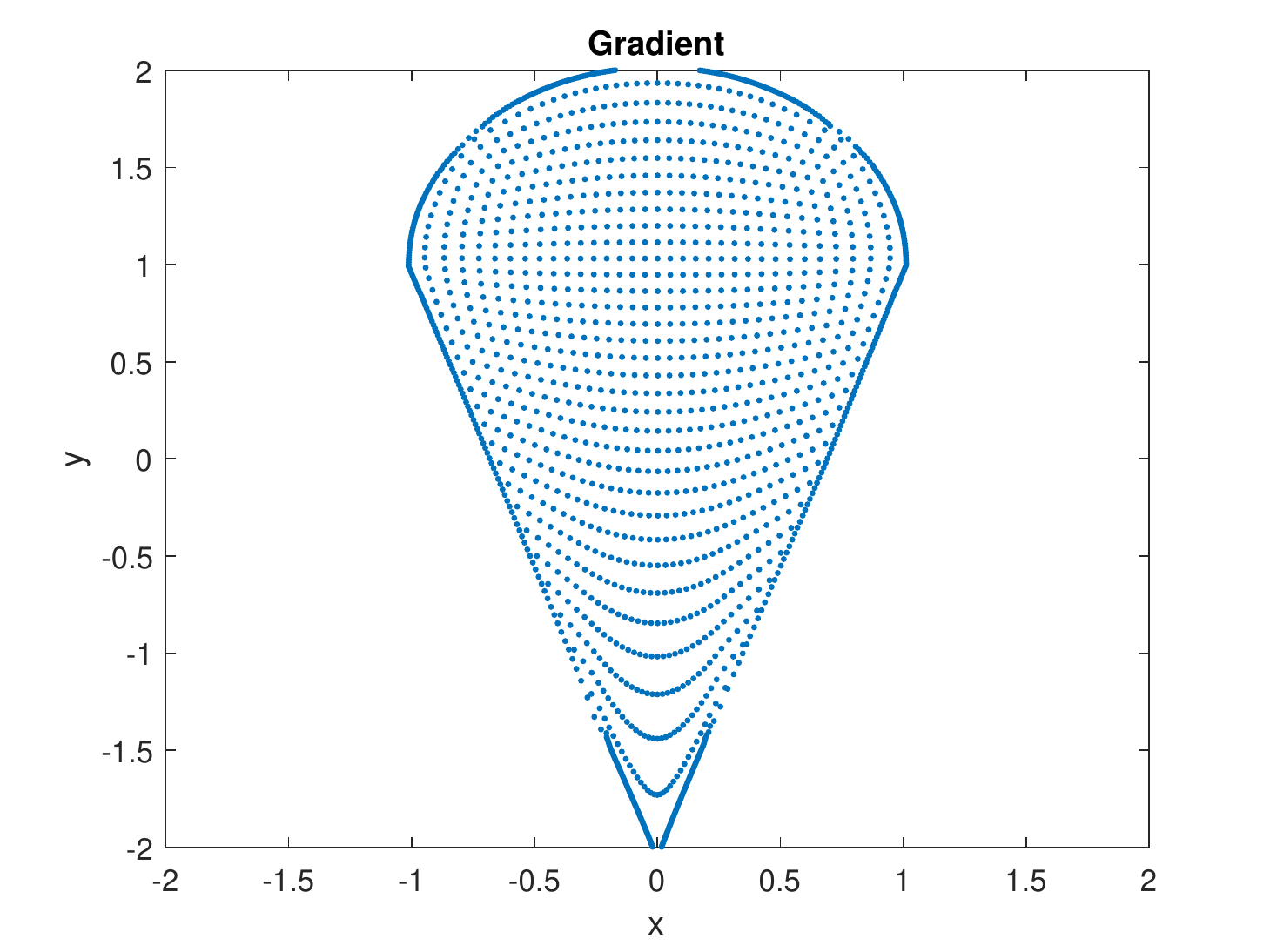}\label{fig:quad2}}
\subfigure[]{
\includegraphics[width=5.715cm,height=54.5mm]{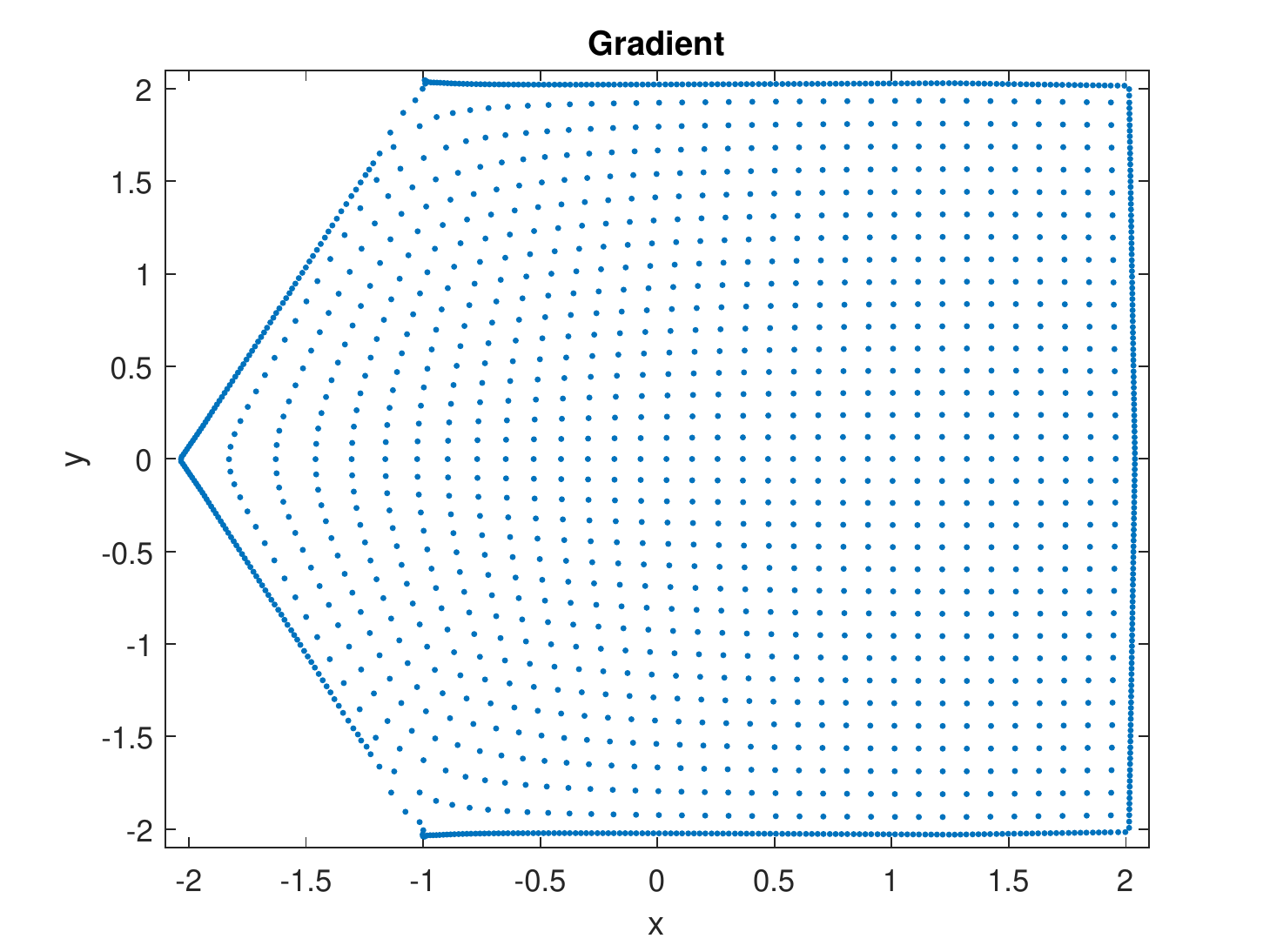}\label{fig:quad1}}
\subfigure[]{
\includegraphics[width=5.715cm,height=54.5mm]{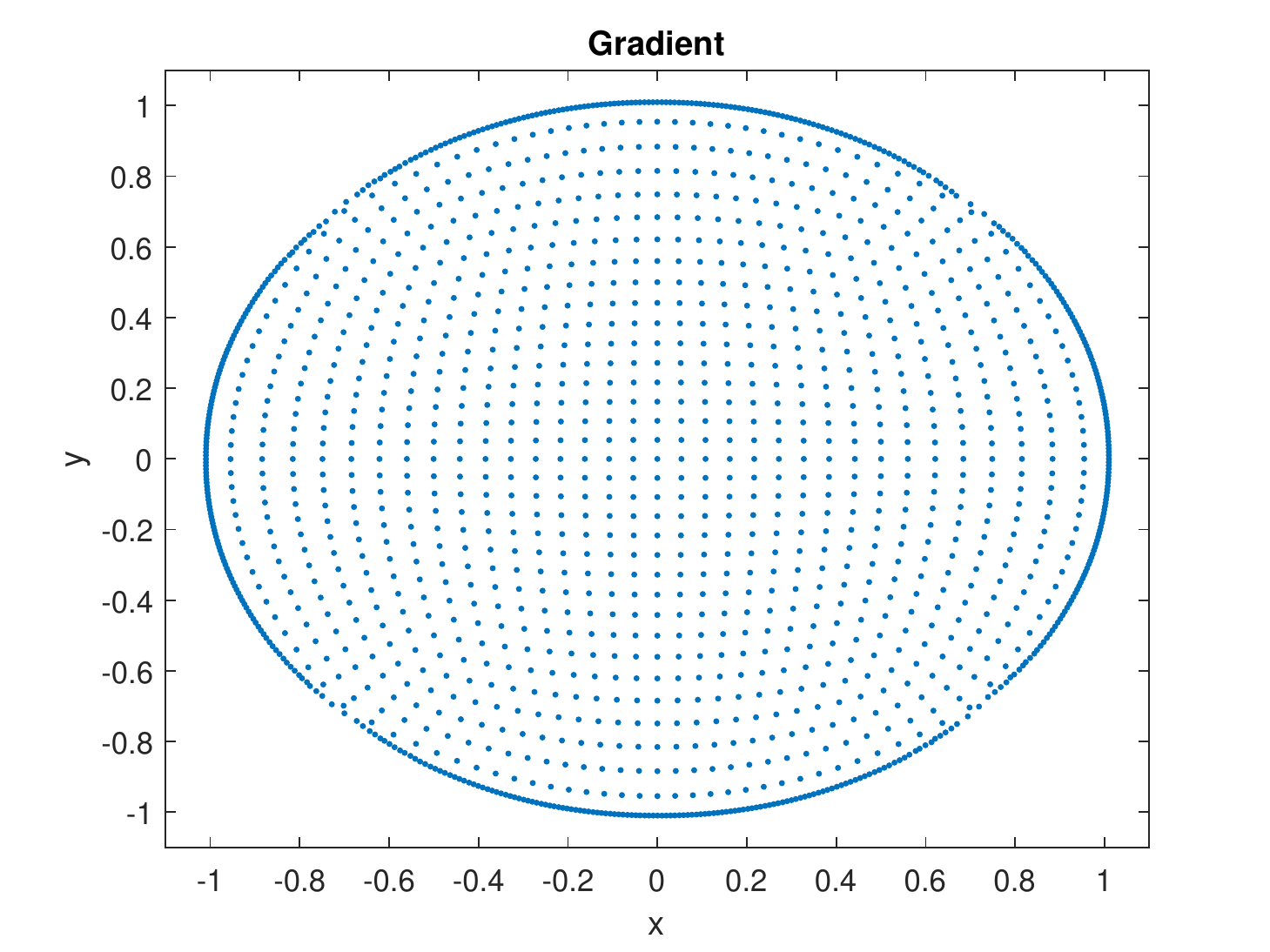}\label{fig:quad2}}
\caption{Computed maps from a square $X$ to various targets $Y$.}\label{fig:shapes}
\end{figure}

Additionally, we consider the case where the domain $X$ is the unit circle and the desired target $Y=(-1.1,1.1)\times(-1.1,1.1)$ is a square.  The computed map is shown in Figure~\ref{CtSfig}, and again achieves the required geometry.

\begin{figure}[!h]
		\begin{center}
\includegraphics[width=5.715cm,height=54.5mm]{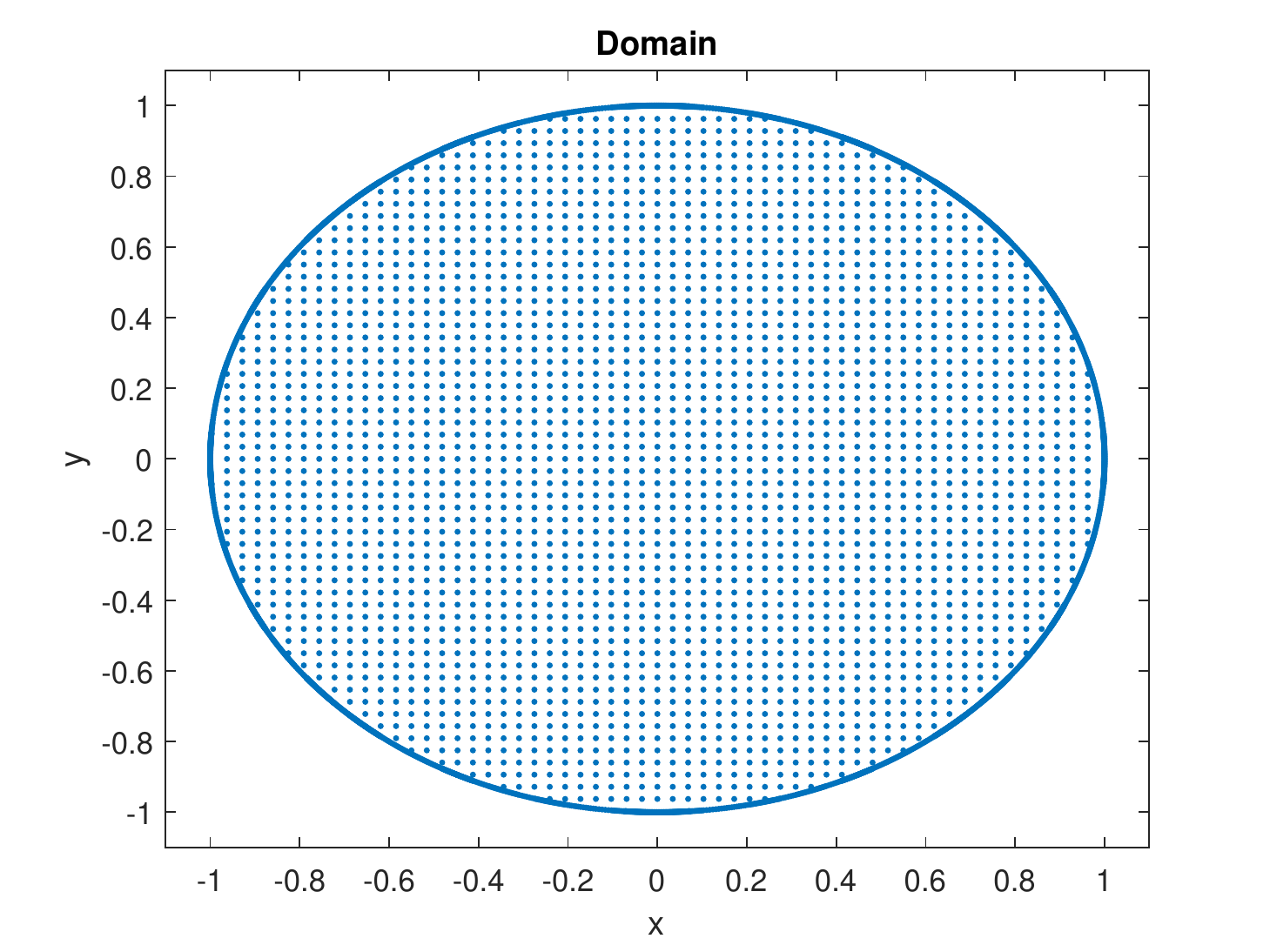}\includegraphics[width=5.715cm,height=54.5mm]{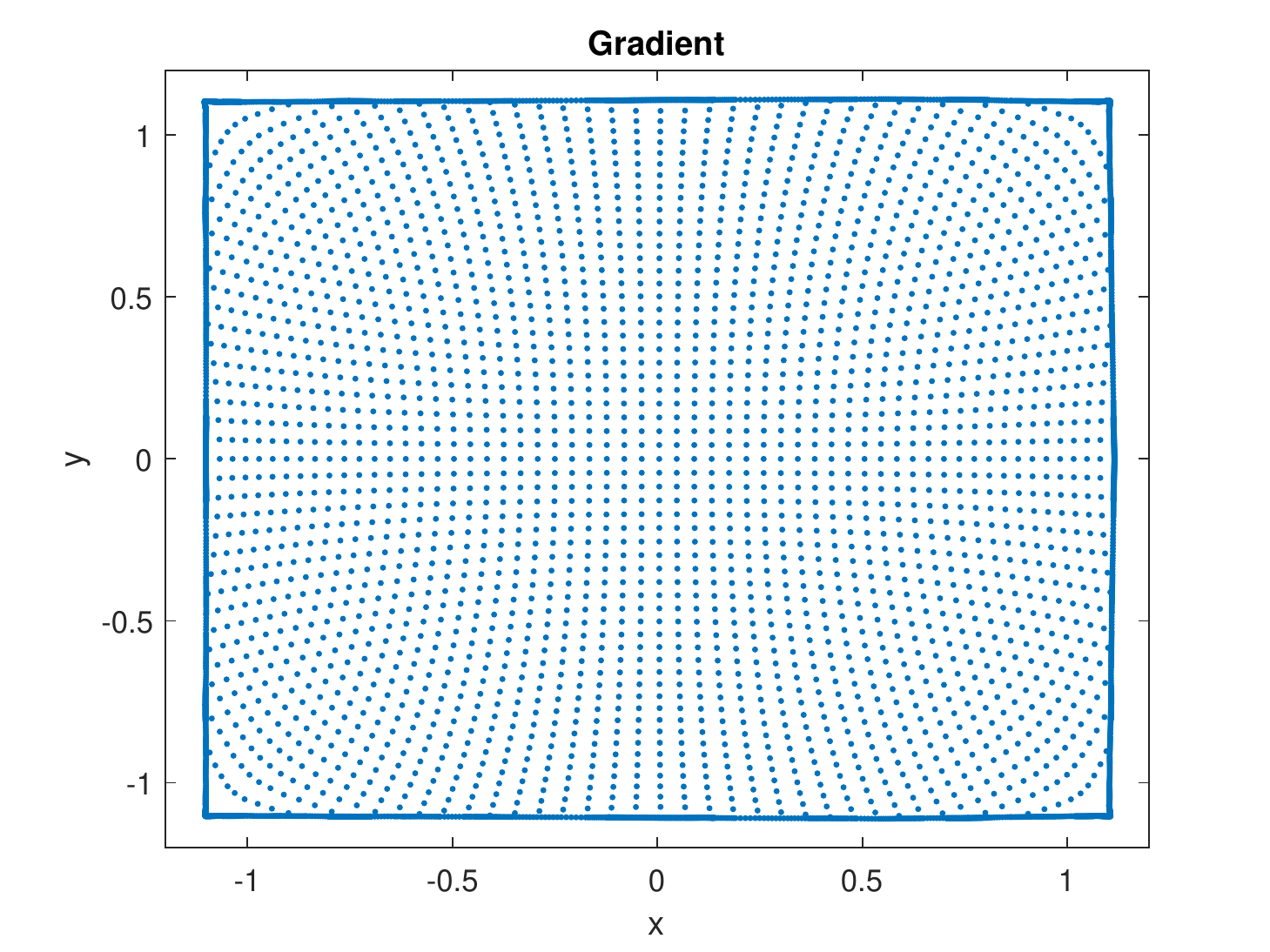}
		\end{center}
\caption{Circular domain $X$ and square target $Y$.}
\label{CtSfig}
\end{figure}

\subsection{Degenerate example}
We also tested our method on the highly degenerate example of a circle $X$ mapped to a line segment $Y = \{0\}\times(-1,1)$. The exact solution for this is $u(x,y) = \frac{x^2}{2}$. This example leads to a highly degenerate PDE that falls outside the purview of our convergence proof.  Nevertheless, our method computes the solution without difficulty, and we again observe $\bO(h)$ convergence.  The error is presented in Table~\ref{degenerateerror} and the computed map is pictured in Figure~\ref{CtL}.

\begin{table}[ht]
\begin{center}
\begin{tabular}{cccc}
$h$ & $\|u^h-\uex\|_\infty$ & Ratio & Observed order \\ \hline
$1.375\times 10^{-1}$ & $9.132 \times 10^{-2}$ &  &  \\
$6.875\times 10^{-2}$ & $3.812 \times 10^{-2}$ & 2.396 & 1.261 \\
$3.438\times 10^{-2}$ & $1.936\times 10^{-2}$ & 1.969 & 0.978 \\
$1.719\times 10^{-2}$ & $1.082\times 10^{-2}$ & 1.790 & 0.840 \\
$8.59\times 10^{-3}$ & $4.636\times 10^{-3}$ & 2.333 & 1.222
\end{tabular}
\end{center}
\caption{Error in mapping a circle to a line segment.}
\label{degenerateerror}
\end{table}

\begin{figure}[ht]
		\begin{center}
\includegraphics[width=5.715cm,height=54.5mm]{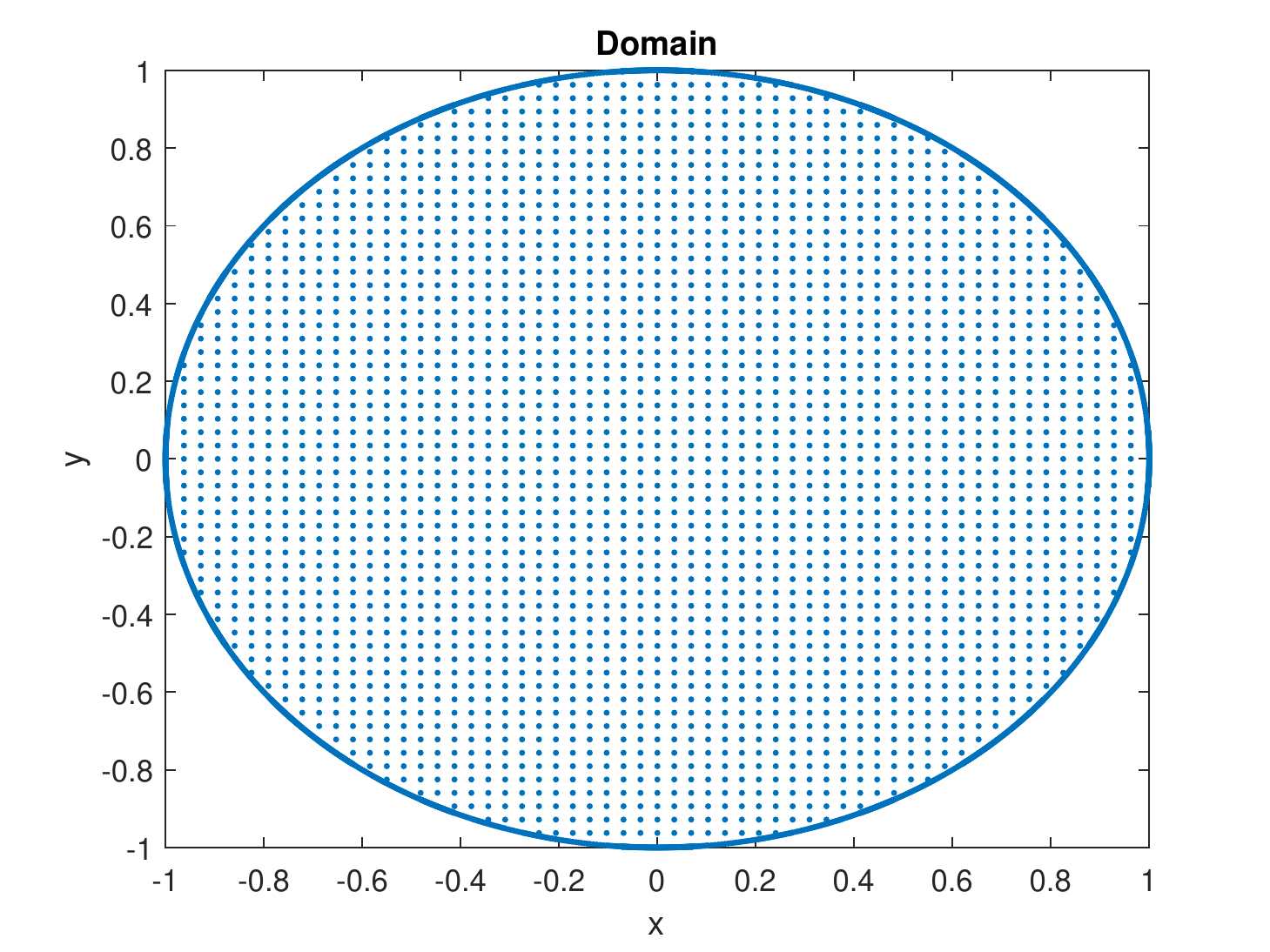}\includegraphics
[width=5.715cm,height=54.5mm]{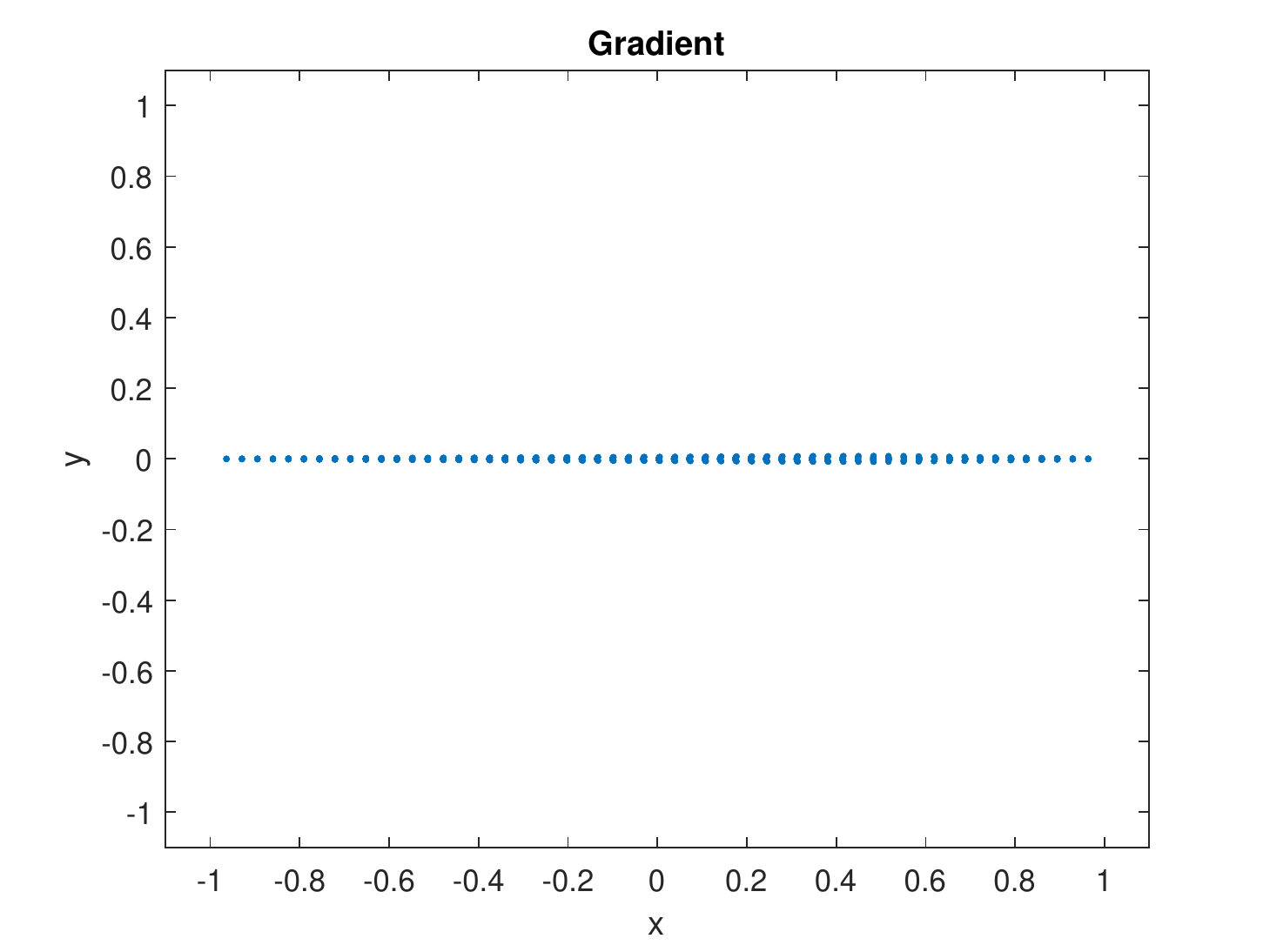}
		\end{center}
\caption{Circular domain $X$ and degenerate target $Y$.}
\label{CtL}
\end{figure}

\section{Conclusion}\label{sec:conclusion}
In this paper, we considered the numerical construction of minimal Lagrangian graphs.  Following~\cite{brendle2010}, we can interpret this as an eigenvalue problem for a fully nonlinear elliptic PDE with a traditional boundary condition replaced by the second type boundary condition.  

To date, the literature has produced very little in the way of numerical analysis for this type of nonlinear eigenvalue problem.  We introduced a numerical framework for solving this problem, which could be easily adapted to more general eigenvalue problems.  This includes a very promising approach for solving PDEs where noisy data fails to exactly solve a required solvability condition or where the discrete solvability condition differs slightly from the solvability condition for the original continuous problem.

We used the monotonicity of our method to demonstrate convergence of the eigenvalue.  By introducing a strong form of stability into our method, we were able to modify the Barles and Souganidis convergence proof to obtain a proof of uniform convergence of our computed solution.  A range of challenging computational examples illustrated the effectiveness of our approach.

\bibliographystyle{plain}
\bibliography{PaperBibliography}
\end{document}